\RequirePackage{luatex85} 
\documentclass[11pt]{amsart}
\usepackage{amssymb}
\usepackage{latexsym}
\usepackage{graphicx}
\usepackage{subfigure}
\usepackage{tikz}
\usepackage[linktocpage=true]{hyperref}
\usepackage[left=3.8cm, right=3.8cm, top=3cm, bottom=3cm]{geometry}
\usepackage{color}
\usepackage{here}
\usepackage[backgroundcolor = lightgray,
linecolor = lightgray,
textsize = tiny]{todonotes}
\usepackage{stackengine}
\usepackage{mathrsfs}
\usepackage[all]{xy}
\usepackage{scrextend}
\usepackage{multirow}

\hypersetup{ colorlinks   = true, 
	urlcolor  = teal, 
	linkcolor    = magenta, 
	citecolor   = teal 
}

\def\e{{\varepsilon}}
\def\beq{\begin{equation}}
	\def\eeq{\end{equation}}

\def\e{\epsilon}

\newcommand{\cC}{{\mathcal C}}
\newcommand{\cD}{{\mathcal D}}

\newcommand{\cW}{{\mathcal W}}
\newcommand{\cWcu}{{\mathcal W}^{cu}}
\newcommand{\cWcs}{{\mathcal W}^{cs}}
\newcommand{\cWc}{{\mathcal W}^{c}}

\newcommand{\cPH}{{\mathcal{PH}}}

\newcommand{\cU}{{\mathcal U}}

\newcommand{\cO}{{\mathcal O}}

\newcommand{\wt}{\widetilde}
\newcommand{\ol}{\overline}

\newcommand{\Z}{{\mathbb Z}}
\newcommand{\R}{{\mathbb R}}

\newcommand{\T}{{\mathbb T}}

\newcommand{\N}{{\mathbb N}}

\newtheorem{thm}{Theorem}[section]
\newtheorem{rem}[thm]{Remark}

\newtheorem{lem}[thm]{Lemma}

\newtheorem{df}[thm]{Definition}
\newtheorem{prop}[thm]{Proposition}
\newtheorem{cor}[thm]{Corollary}
\newtheorem{claim}[thm]{Claim}

\newtheorem*{main thm}{Main Theorem}
\newtheorem*{main thm bis}{Main Theorem (alternate version)}

\newtheorem{theoalph}{Theorem}

\providecommand{\norm}[1]{\lVert#1\rVert}

\begin{document}
\numberwithin{equation}{section}
	
\title[Some hyperbolicity revisited]{Some hyperbolicity revisited and robust transitivity}

\author[Luis Pedro Pi\~{n}eyr\'ua]{Luis Pedro Pi\~{n}eyr\'ua$^{*}$}
\thanks{$^{*}$L.P.P was partially supported by CAP's doctoral scholarship and CSIC group 618}


\subjclass[2010]{Primary: 37D30. Secondary: 37B05.
}

\begin{abstract}
	In this article we revisit the notion of \textit{Some Hyperbolicity} introduced by Pujals and Sambarino in \cite{PuSa}. We present a more general definition, that in particular can be applied to the symplectic context (something that the previous couldn't). As an application we construct $C^1$ robustly transitive derived from Anosov diffeomorphisms with mixed behaviour on center leaves.
\end{abstract}
	
\maketitle

\section{Introduction}
In short, dynamical system theory is the study of motion and we want to understand the behaviour of most orbits. Typically the structure of the orbits is very complicated, for example in some cases there are orbits that almost fill the whole space, making it indecomposable from the dynamical point of view. That is what is called \textit{transitivity}: a dynamical system is said to be \textit{transitive} if it has a dense forward orbit. Even more interesting are the systems that present a dynamical feauture that is stable or robust (meaning that it persists under perturbation). We say that a dynamical system is \textit{robustly transitive}, if there is a neighbourhood of the system (in some particular topology) such that every system in this neighbourhood is transitive.

The first example of a $C^1$ robustly transitive diffeomorphism was given by D. A. Anosov in \cite{An}, where he proved that uniformly hyperbolic diffeomorphisms (today called Anosov diffeomorphisms) are stable under $C^1$ perturbations. As a corollary every transitive Anosov diffeomorphism, is in fact $C^1$ robustly transitive. Years later M. Shub \cite{Sh} constructed the first non-Anosov $C^1$ robustly transitive diffeomorphism on the torus $\T^4$ and a few years later R. Ma\~{n}\'e improved this result and introduced an example on $\T^3$ \cite{Ma1}. Both Shub's and Ma\~{n}\'e's examples are isotopic to linear Anosov diffeomorphisms and by that reason they're called \textit{derived from Anosov} examples (from now on DA diffeomorphisms). Another way to construct $C^1$ robustly transitive diffeomorphisms was introduced by C. Bonatti and L. D\'iaz in \cite{BD}. Their technique is based on the existence of some particular hyperbolic subsets called \textit{blenders}. With this geometric approach, the authors were able to build examples $C^1$-close to time-$t$ maps of Anosov flows (hence, isotopic to the identity) as well as examples $C^1$-close to the product of Anosov times the identity (therefore, with trivial action on the center). All these non-hyperbolic examples are partially hyperbolic (although there are also $C^1$ robustly transitive examples that are not partially hyperbolic \cite{BV}).

In \cite{PuSa} Pujals and Sambarino introduced the SH Property (\textit{Some hyperbolicity}) for partially hyperbolic diffeomorphisms. This property, which is $C^1$ robust, in addition to minimality of the strong stable foliation implies $C^1$ robust minimality of the strong stable foliation, therefore $C^1$ robust transitivity. As an application of this approach, they re-obtained the examples of Shub and Ma\~{n}\'e. 

Our contribution in this article is the introduction of a more general concept of SH Property, that we called \textit{SH-Saddle property}. This new definition is a natural generalization of the previous SH definition and as a consequence it can be applied to a larger number of cases. In particular, it has the advantage of being applicable in the symplectic context (something that the previous definition couldn't). 

Let us be more precise. We say that a diffeomorphism $f:M \to M$ is \textit{partially hyperbolic} if there exists a nontrivial $Df$-invariant splitting $TM=E_f^{ss} \oplus E_f^c \oplus E_f^{uu}$ of the tangent bundle and numbers $\lambda _s,\lambda _c^-, \lambda _c^+, \lambda_u$ with $\lambda_s <1< \lambda_u$ and $\lambda_s < \lambda_c^- < \lambda_c^+< \lambda_u$ such that: 
\begin{equation*}
	\|D f_x|_{E^{ss}_f}\|<\lambda_s, \qquad  \lambda_c^- <\  \|D f_x|_{E^c_f}\|< \lambda_c^+ , \qquad  \lambda_u   <\  m(D f_x|_{E^{uu}_f}). 
\end{equation*}
We will denote by $\cPH(M)$ to the set of all partially hyperbolic diffeomorphisms of $M$. It is well known that the strong bundles $E_f^{uu}$ and $E_f^{ss}$ integrate into unique invariant foliations $\cW^{uu}_f$ and $\cW^{ss}_f$ respectively, called the \textit{strong unstable} and \textit{strong stable} foliations \cite{HPS}. For $* = uu,ss$, and  for any $x \in M$, we denote by $\cW_f^*(x)$ the leaf of $\cW_f^*$ through $x$. In the following, for any $*\in \{ss,uu\}$, we denote by $d_{\mathcal{W}_f^*}$ the leafwise distance, and for any $x \in M$ and for any  $\epsilon>0$, we denote by 
$\cW_f^*(x,\varepsilon):=\{y \in \cW_f^*(x): d_{\cW_f^*}(x,y)< \varepsilon\}$ the $\varepsilon$-ball in $\cW_f^*$ of center $x$ and radius $\varepsilon$. 

Now given a $\mathbb{R}$-vector space $V$ with an inner product, we say that a \textit{cone} in $V$ is a subset $\cC$ such that there is a non-degenerate quadratic form $B:V \to \mathbb{R}$ such that $\cC=\{v\in V: B(v)\leq 0\}$. Analogously we can express the cone $\cC$ according to a decomposition $V=E\oplus F$: 
\begin{equation*} \label{eqcono1}
	\cC=\{v=(v_E,v_F): \norm{v_E}\leq \theta \norm{v_{F}}\}
\end{equation*} for some $\theta>0$. In this case we observe that $B(v)=-\theta^2\norm{v_{F}}^2+\norm{v_{E}}^2$. We are going to say that the number $\theta$ in the equation above is the \textit{size} of the cone. In some cases we will note by $\cC_{\theta}$ instead of $\cC$ to make emphasis on the size of $\cC$. The \textit{dimension of a cone} is the maximal dimension of any subspace contained in the cone. 

Finally, given $f\in \cPH(M)$, we are going to say that a \textit{$d$-center cone} in $x\in M$ is simply a cone $\cC(x)$ in $E^{c}_f(x)$ of dimension $d\leq \text{dim}E^c_f$. 

We now introduce the main definition of the article.

\begin{df}[SH-Saddle property for unstable foliations] \label{defshuu}
	Given $f\in \cPH(M)$ we say that the strong unstable foliation $\cW^{uu}_f$ has the $SH$-Saddle property of index $d\leq \textnormal{dim}E^c_f$ if there are constants $L>0$, $\theta>0$, $\lambda_0>1$ and $C>0$ such that the following hold. For every point $x\in M$, there is a point $y\in \cW^{uu}_f(x,L)$ such that:
	\begin{enumerate}
		\item \label{sh1} There is a $d$-center cone field of size $\theta$ along the forward orbit of $y$ which is $Df$-invariant, i.e. there exist $\cC^u_{\theta}(f^l(y))\subset E^c_f(f^l(y))$ such that $Df(\cC^u_{\theta}(f^l(y)))\subset \cC^u_{\theta}(f^{l+1}(y))$ for every $l\geq 0$. \label{c1shu}
		\item $\norm{Df^n_{f^l(y)}(v)}\geq C\lambda_0^n\norm{v}$  for every $v\in \cC^u_{\theta}(f^{l}(y))$ and every $l,n\geq0$. \label{c2shu}
	\end{enumerate}
\end{df}

Notice that if the strong unstable foliation has SH-Saddle property of index $d=\text{dim}E^c_f$, we get the original definition of SH Property introduced in \cite{PuSa}. The only difference is that we express the uniform expanding behaviour in the center bundle in terms of a cone criterion. This allows us to treat the case where there is expansion in a subbundle of the center subspace instead of the whole center bundle. Moreover since properties that are presented in terms of cones are in general robust in the $C^1$ topology, we are able to prove that this new SH-Saddle property is $C^1$ open among partially hyperbolic diffeomorphisms (Theorem \ref{shisopen}). 

We can make an analogous definition of SH-Saddle property for the strong stable foliation. In this case we ask for the invariance of the cones for the past. 
\begin{df}[SH-Saddle property for stable foliations] \label{defshss}
	Given $f\in \cPH(M)$ we say that the strong stable foliation $\cW^{ss}_f$ has the $SH$-Saddle property of index $d\leq \textnormal{dim}E^c_f$ if there are constants $L>0$, $\theta>0$, $\lambda_{0}>1$ and $C>0$ such that the following hold. For every point $x\in M$, there is a point $y\in \cW^{ss}_f(x,L)$ such that: 
	\begin{enumerate}
		\item There is a $d$-center cone field of size $\theta$ along the backward orbit of $y$ which is $Df^{-1}$-invariant, i.e. there exist $\cC^s_{\theta}(f^l(y))$ such that $Df^{-1}(\cC^s_{\theta}(f^l(y)))\subset \cC^s_{\theta}(f^{l-1}(y))$  for every $l \leq 0$. \label{c1shs}
		\item $\norm{Df^n_{f^l(y)}(v)}\geq C\lambda_0^{-n}\norm{v}$  for every $v\in \cC^s_{\theta}(f^l(y)))$ and every $l,n\leq 0$. \label{c2shs}
	\end{enumerate}
\end{df}

With this new approach we first give a sufficient condition for a DA diffeomorphism to be $C^1$ robustly transitive (Theorem \ref{teorobtran}). As an application of this result, we are able to build new $C^1$ robustly transitive DA diffeomorphisms, in particular with any center dimension and with as many different behaviours on center leaves as desire. Moreover, these examples can be made in a way such that they have mixed behaviour on center leaves. In particular they present a dominated splitting that is not coherent with the hyperbolic splitting of their linear Anosov part, a difference with its predecessors DA examples (\cite{BD}, \cite{BV}, \cite{Ma1} \& \cite{Sh}).

\begin{theoalph} \label{teoejemixed}
	Let $n\geq 4$, let $A\in \textnormal{SL}(n,\Z)$ be a hyperbolic symmetric matrix with a splitting of the form $\R^n=E^{ss}_A\oplus E^{ws}_A\oplus E^{wu}_A \oplus E^{uu}_A$. Denote by $E^c_A=E^{ws}_A\oplus E^{wu}_A$ and let $k=\textnormal{dim}E^c_A\geq 2$. Then there exist a $C^1$ robustly transitive partially hyperbolic diffeomorphism $f:\T^n\to \T^n$, isotopic to $A$  with a splitting of the form $T\T^n=E^{ss}_f\oplus E^c_f\oplus E^{uu}_f$ such that $\textnormal{dim}E^*_f=\textnormal{dim}E^*_A$ for $*=ss,c,uu$, and with $k+1$ fixed points $p_0,p_1,\dots,p_k$ such that: 
	$\text{index}(p_j)=j+\textnormal{dim}E^{ss}_A$ for every $j=0,\dots,k$.
	
	Moreover the center bundle $E^c_f$ does not admit a dominated splitting. In particular the splitting of $f$ is not coherent with the hyperbolic splitting of $A$. 
\end{theoalph}

We remark that in the theorem above $E^c_A$ is strictly hyperbolic but the proof for the case where $E^c_A$ is entirely contracting or expanding works as well. In these last cases, the result is basically contained in \cite{PuSa} with the difference of dealing with the minimality of the strong unstable foliation instead of transitivity.

Let us mention that in his PhD thesis R. Potrie \cite{Po} (page 152) constructed a $C^1$ robustly transitive example on $\T^3$ but in this case, the example's dominated splitting is not coherent with its Anosov part, although the definition of partial hyperbolicity here is a bit different. Recently P. Carrasco and D. Obata showed in \cite{CaOb} that the example introduced in \cite{BC} is $C^1$ robustly transitive. This example although it is a skew product on $\T^4$, it has the particularity of having mixed behaviour on the center (which is two-dimensional) and therefore makes it a new example. The authors mention in the paper that their example can't have the SH Property (the original version). However, it follows directly from the calculations in their article, that the example has the SH-Saddle property. Therefore, at the moment every known example of $C^1$ robustly transitive partially hyperbolic diffeomorphism verifies the SH-Saddle property of some index.

Besides the examples in the Theorem \ref{teoejemixed}, we also present two additional examples of $C^1$ robustly transitive partially hyperbolic diffeomorphisms. These examples are in a sense similar to the ones in Theorem \ref{teoejemixed} but with a different flavor. The first one has the particularity of being symplectic and it has the SH-Saddle property. Recall that the original SH property is incompatible with being symplectic, so this example shows that SH-Saddle property is useful in the symplectic context. The second example is different to the last ones since the set of points of the manifolds where hyperbolicity fails is not localized in small neighborhoods of fixed points.  

Finally let us mention that the \textit{robust} minimality of the strong foliations, which is the main purpose of the original SH definition in \cite{PuSa}, is out of reach, since the same strategy Pujals and Sambarino made is not adaptable to the saddle case, and some new approach is needed. In fact, just the minimality of the strong foliations for a DA diffeomorphism, like the ones in Theorem \ref{teoejemixed}, is not easy to get, since even for true Anosov diffeomorphisms this is a very difficult problem. Recently it was announced by Avila-Crovisier-Eskin-Potrie-Wilkinson-Zhang that $\mathcal{W}^{uu}$ is minimal for every $C^{1+\alpha}$ Anosov diffeomorphism of $\T^3$.

We end this introduction by presenting a few questions that at the moment we don't know the answers. 
\newline \,
\newline
\textbf{Question 1.} \textit{Is it possible to obtain a criterion for the $C^1$ robust minimality of the strong foliations in the SH-Saddle case, as the one obtained by Pujals and Sambarino in \cite{PuSa}?}
\newline \,
\newline
\textbf{Question 2.} \textit{Does every $C^1$ robustly transitive partially hyperbolic diffeomorphism verify the SH-Saddle property of some index?}
\newline\,
\newline 
\textbf{Question 3.} \textit{Does transitivity in addition to SH-Saddle property imply $C^1$ robust transitivity?}
\newline

\subsection*{Organization of the paper} 
In Section \ref{sectionshproperty} we prove tha the SH-Saddle property is a $C^1$ open condition among partially hyperbolic diffeomorphisms. 
In Section \ref{sderivedfromanosov} we present a criterion for DA diffeomorphisms that guarantees $C^1$ robust transitivity. Finally in Section \ref{ssymplexample} we apply the previous results to build new DA examples and prove Theorem \ref{teoejemixed}. 

\section*{Acknowledgments}

This work is part of the author's PhD thesis, and he would like to thank his advisors Rafael Potrie and Mart\'in Sambarino, for various discussions and suggestions about this work. The author also wish to thank Enrique Pujals for many conversations as well as the hospitality of the Graduate Center of the CUNY, where this work begun. Finally the author wish to thank the anonymous referees for their careful reading and suggestions that helped us to improve the article.

\section{SH-Saddle property is $C^1$ open} \label{sectionshproperty} 

The definitions \ref{defshuu} and \ref{defshss} of SH-Saddle property are given for the strong unstable and strong stable foliations respectively. In many parts of the article we'll need the presence of the two simultaneously, thus for simplicity we give the following definition of SH-Saddle property for diffeomorphisms by grouping together these two. 

\begin{df}[SH-Saddle property for diffeomorphisms] \label{defshf}
	We say that $f\in \cPH(M)$ has $(d_1,d_2)$ SH-Saddle property if the following conditions hold:
	\begin{enumerate} 
		\item $\cW^{ss}_f$ has the $SH$-Saddle property of index $d_1$.
		\item $\cW^{uu}_f$ has the $SH$-Saddle property of index $d_2$.
	\end{enumerate}
\end{df}	

\begin{rem}
	Notice that not necessarily we have $d_1+d_2=\textnormal{dim}E^c_f$, in fact in many cases we are going to have $d_1+d_2<\textnormal{dim}E^c_f$. For simplicity in some parts of the article, we are going to omit the indexes $(d_1,d_2)$ and we're just going to say that a partially hyperbolic diffeomorphism has the SH-Saddle property. 
\end{rem}

\begin{rem}
	The SH-Saddle property does not depend on the choice of the Riemannian metric.
\end{rem}




In consequence of the previous remark we get the following fact.

\begin{prop} \label{fshifffksh}
	A partially hyperbolic diffeomorphism $f$ has the SH-Saddle property if and only if $f^N$ has the SH-Saddle property for some $N\in \N$.
\end{prop}	

Let us introduce some notation that will be useful along the article, and will help us to get a better understanding of what it means the SH-Saddle property. Let $f\in \cPH(M)$ be such that its unstable foliation has the SH-Saddle property of index $d\leq \text{dim}E^c_f$ and let $L>0$, $\theta>0$, $\lambda_0>1$ and $C>0$ be the constants given by Definition \ref{defshuu}. We can define the following subset:
\begin{equation*} \label{defH+}
	H^+_{\lambda_0,d}(f)=\{y\in M: \text{conditions \ref{c1shu} and \ref{c2shu} of Definition \ref{defshuu} are satisfied}\}.
\end{equation*}  
Then the unstable foliation has the SH-Saddle property of index $d$ if and only if $$H^+_{\lambda_0,d}(f)\cap \cW^{uu}_f(x,L)\neq \emptyset \  \ \text{for every} \ x\in M.$$ In the same way let $f \in \cPH(M)$ be such that its stable foliation has the SH-Saddle property of index $d$ and let $L>0$, $\theta>0$, $\lambda_0>1$ and $C>0$ be the constants given by Definition \ref{defshss}, then we can define the following subset:
\begin{equation*}\label{defH-}
	H^{-}_{\lambda_0,d}(f)=\{y\in M: \text{conditions \ref{c1shs} and \ref{c2shs} of Definition \ref{defshss} are satisfied}\}.
\end{equation*} 
and the stable foliation has the SH-Saddle property of index $d$ if and only if  $$H^-_{\lambda_0,d}(f)\cap \cW^{ss}_f(x,L)\neq \emptyset \ \ \text{for every} \ x\in M.$$ 

\begin{rem}
	The sets $H^{*}_{\lambda_0,d}(f)$ are closed subsets of $M$, for $*=+,-$.
\end{rem}

	In the reminder of this section we are going to prove that the SH-Saddle property is $C^1$ open among $\cPH(M)$. According to Definition \ref{defshf} we only have to prove that having an unstable manifold with SH-Saddle property \ref{defshuu}, and having a stable manifold with SH-Saddle property \ref{defshss} are $C^1$ open properties. We are going to focus on the unstable case, since the stable case is completely symmetric. We begin with a few simple lemmas that only uses the properties of the $C^1$ topology.

	\begin{lem} \label{lextensionconos1}
		
		Suppose that the unstable foliation of $f\in \cPH(M)$ has SH-Saddle property of index $d$. Then there is $\delta_0>0$ such that if $d(y,H^+_{\lambda_0,d}(f))<\delta_0$, then $y$ and $f(y)$ have  $d$-center cones $\cC^u(y)$ and $\cC^u(f(y))$ such that $Df(\cC^u(y))\subseteq \cC^u(f(y))$. 
	\end{lem}
	
	\begin{proof}
		We know that for every $x\in H^+_{\lambda_0,d}(f)$ there is a cone $\cC^u(x)$ which is $Df$-invariant. Now for the first part of the lemma just notice that since the family of center cones comes from a non-degenerate quadratic form, we can extend this quadratic form to neighbours by continuity. For the invariance just observe that $Df$ is uniformly continuous.
	\end{proof}
	Since the family of cones varies continuously, the same family of cones in the lemma above is still invariant for every $g$ sufficiently $C^1$ close to $f$. Then we obtain the following.
	\begin{lem} \label{lextensionconos2}
		Suppose that the unstable foliation of $f\in \cPH(M)$ has the SH-Saddle property of index $d$, and let $\delta_0>0$ be as in Lemma \ref{lextensionconos1}. Then there is a $C^1$-neighbourhood $\cU_0(f)$ of $f$ such that if $g\in \cU_0(f)$ and $d(y,H^+_{\lambda_0,d}(f))<\delta_0$ then $Dg(\cC^u(y))\subseteq \cC^u(g(y))$. 
	\end{lem}
	
	Now we are ready to prove the main theorem of this section.
	
	\begin{thm}\label{shisopen}
		Suppose that the unstable foliation of $f\in \cPH(M)$ has SH-Saddle property of index $d$. Then there are constants $\lambda>1$, $L>0$ and a  $C^1$-neighbourhood $\mathcal{V}$ of $f$ such that, if $g\in \mathcal{V}$ then
		$H^+_{\lambda,d}(g)\cap \cW^{uu}_g(x,L)\neq \emptyset$ for every $x\in M$ (i.e.: the unstable foliation $\cW^{uu}_g$ has the SH-Saddle property of index $d$ with constants $\lambda>1$ and $L>0$). 
	\end{thm}
	
	\begin{proof}
		Take $f\in \cPH(M)$ such that its strong unstable foliation has the SH-Saddle property of index $d$. That means there are constants $\lambda_0>1$, $L_0>0$ and $C>0$ such that Definition \ref{defshuu} holds. Then we have:
		$$H^+_{\lambda_0,d}(f)\cap \cW^{uu}_f(x,L_0)\neq \emptyset \ \text{ for every } x\in M.
		$$ Let $\delta_0>0$ and $\cU_0(f)$ be as in Lemma \ref{lextensionconos1} and Lemma \ref{lextensionconos2}. Take $c>0$ such that $\frac{\lambda_{0}}{1+c}= \lambda_{1}>1$. Take $\e>0$, $\delta_1\in (0,\delta_0)$ and $\mathcal{U}_{1}(f)\subseteq \mathcal{U}_{0}(f)$ such that if $g\in \mathcal{U}_{1}(f)$, $d(x,y)<\delta_1$ and $v\in T_xM$ has $\norm{v}=1$ then:
		\begin{equation*}
			\norm{Df_x(v)-Dg_y(w)}<\e
		\end{equation*}
		where $w=P_{x,y}(v) \in T_yM$ is the parallel transport of $v$ from $x$ to $y$. 
		We can take $\e>0$ small enough such that if $d(x,y)<\delta_1$ and $g\in \mathcal{U}_{1}(f)$ then:
		\begin{equation} \label{constantdelta1}
			\frac{1}{1+c}\leq \frac{\norm{Df_x}}{\norm{Dg_y}} \leq 1+c \ \textnormal{ and } \ \frac{1}{1+c}\leq \frac{m\{Df_x\}}{m\{Dg_y\}} \leq 1+c.
		\end{equation}
		Finally let $K^+=\sup \{\norm{Df|_{E^c(x)}}: x\in M\}$ and $K^-=\inf \{ m\{Df|_{E^c(x)}\}: x\in M\}$. We can assume that $K^+$ and $K^-$ are $C^1$-uniform on a neighbourhood $\cU_2(f)\subseteq \cU_1(f)$.
		
		Let $m_1\in \Z^+$ be large enough such that 
		\begin{equation} \label{m2}
			(\lambda_u)^{m_1}>2 
		\end{equation} and moreover, for any $g\in \cU_2(f)$ and any $x\in M$ we have
		\begin{equation} \label{mepsilon}
			\cW^{uu}_g(g^{m_1}(x),L_0) \subset g^{m_1}(\cW^{uu}_g(x,\delta_1/4)).
		\end{equation}  
		Now take $m_2\in \Z^+$ sufficiently large, and take $\lambda_2$ such that
		\begin{equation} \label{eqlambda}
			C\lambda_1^{m_2}(K^-)^{m_1}\geq \lambda_2>1.
		\end{equation}
		Let $\cU_3(f)$ and $\delta_2\in (0,\delta_1/2)$ be such that if $d(x,y)<\delta_2$ and $g\in \cU_3(f)$, then $d(f^j(x),g^j(y))<\delta_1$, for $0\leq j \leq m_2$. 
		
		Finally take $\cU_4(f) \subset \cU_3(f)$ such that for every $g \in \cU_4(f)$ we have 
		\begin{equation*}
			d_H(\cW^{uu}_g(x,L_0),H^+_{\lambda_0,d}(f))<\delta_2.
		\end{equation*}
		We claim that every $g\in \mathcal{V}=\cU_4(f)$ has unstable manifold with SH-Saddle property of index $d$. In fact, we are going to see that $g^{k_0}$ has this property for $k_0=m_1+m_2$, with constants $L=2L_0$ and $\lambda_2>1$ (where $\lambda_2$ comes from Equation \eqref{eqlambda}). Then we conclude by Proposition \ref{fshifffksh}.
		
		To see this, take $g\in \mathcal{V}$ and $x\in M$. We know there are points $x^u_0 \in H^+_{\lambda_0,d}(f)$ and $z^u_0 \in \cW^{uu}_g(x,L_0)$ such that $d(x^u_0,z^u_0)<\delta_2$. Notice that since $\delta_2<\delta_0$ we know there is a center cone $\cC^u(z^u_0)$.
		
		Now let $v \in \cC^u(z^u_0)$. Since  $d(x^u_0,z^u_0)<\delta_2$ we have that $d(f^j(x^u_0),g^j(z^u_0))<\delta_1$ for $0\leq j \leq m_2$. Then we have:
		\begin{equation} \label{cexpansion1}
			\norm{Dg^{m_2}_{z^u_0}(v)}\geq \frac{\norm{Df^{m_2}_{x^u_0}(w)}}{(1+c)^{m_2}} \geq C \left(\frac{\lambda_0}{1+c}\right)^{m_2}\norm{w}=C\lambda_1^{m_2}\norm{w} 
		\end{equation}
		where $w=P_{z^u_0,x^u_0}(v)$ is the parallel transport of $v$ from $z^u_0$ to $x^u_0$. Now, 
		\begin{equation} \label{cexpansion2}
			\norm{Dg^{k_0}_{z^u_0}(v)}=\norm{Dg^{m_1}_{g^{m_2}(z^u_0)}(Dg^{m_2}_{z^u_0}(v))} \geq (K^-)^{m_1}C\lambda_1^{m_2} \norm{v} \geq \lambda_2 \norm{v}.
		\end{equation}	
		Now by \eqref{mepsilon}, we can apply the same argument to $\cW^{uu}_g(g^{k_0}(z^u_0),L_0)$, and we can find points $x^u_1\in H^+_{\lambda_0,d}(f)$ and $z^u_1 \in \cW^{uu}_g(g^{k_0}(z^u_0),L_0)$ such that $d(x^u_1,z^u_1)<\delta_2$. Then, there is a center cone $\cC^u(z^u_1)$ and for every vector $v\in \cC^u(z^u_1)$ we have $\norm{Dg^{k_0}_{z^u_1}(v)}\geq \lambda_2\norm{v}$. Call $y^u_1=g^{-k_0}(z^u_0)$. Now, by \eqref{mepsilon} we have that $g^{-m_1}(z^u_1)\in \cW^{uu}_g(g^{m_2}(y^u_1),\delta_1/4)$ and this implies that
		$$d(x^u_0,y^u_1)\leq d(x^u_0,y^u_0)+d(y^u_0,y^u_1)<\delta_2 +\frac{\delta_1}{4}\leq \frac{\delta_1}{2}+\frac{\delta_1}{4}<\delta_1 < \delta_0$$ and there is a $d$-center cone $\cC^u(y^u_1)$. Moreover we have that 
		$$ d(g^j(y^u_1),g^j(z^u_0))<\delta_1 \ \ \text{for every} \ \ 0\leq j \leq m_2 
		$$
		and by applying the same calculations as in \eqref{cexpansion1} and \eqref{cexpansion2} we have 
		\begin{equation*}
			\norm{Dg^{2k_0}_{y^u_1}(v)}\geq (\lambda_2)^2\norm{v}.
		\end{equation*} 	
		Inductively, we can find sequences $\{z^u_n\}_{n\in \N}$, $\{y^u_n\}_{n\in \N}$ and $\{x^u_n\}_{n\in \N}$, which verify the following:
		\begin{itemize}
			\item $z^u_{n}\in \cW^{uu}_g(g^{k_0}(z^u_{n-1}),L_0)$.
			\item $x^u_n\in H^+_{\lambda_0,d}(f)$.
			\item $d(z^u_n,x^u_n)<\delta_2$.
			\item $y^u_n=g^{-k_{0}n}(z^u_n)$.
		\end{itemize} 
		Notice that since $z^u_{n}\in \cW^{uu}_g(g^{k_0}(z^u_{n-1}),L_0)$, we have $d(g^{-k_0}(z^u_n),z^u_{n-1})<\delta_1/4$ by \eqref{mepsilon}. Then by \eqref{m2} we have
		\begin{eqnarray*}
			d(y^u_{n-1},y^u_n)&=&d(g^{-k_0(n-1)}(z^u_{n-1}),g^{-k_0n}(z^u_n)) \\
			&=&d(g^{-k_0(n-1)}(z^u_{n-1}), g^{-k_0(n-1)}(g^{-k_0}(z^u_n)) 
			\leq \frac{\delta_1}{4} \left( \frac{1}{2} \right)^{n-1} =\delta_1 \left( \frac{1}{2} \right)^{n+1}.
		\end{eqnarray*}
		By the triangular inequality, the distance between $x^u_0$ and $y^u_n$ is 
		\begin{equation*}
			d(x^u_0,y^u_n) \leq d(x^u_0,y^u_0)+\sum_{j=1}^{n} d(y^u_{j-1},y^u_j)\leq \delta_2+\sum_{j=1}^{n} \delta_1 \left(\frac{1}{2} \right)^{j+1} < \sum_{j=0}^{n} \left(\frac{1}{2} \right)^{j+1}\delta_1 <\delta_1
		\end{equation*}
		since $\delta_2 <\delta_1/2$. 
		Then there is a $d$-center cone $\cC^u(g^j(y^u_n))$ such that $Dg(\cC^u(g^j(y^u_n))) \subset \cC^u(g^{j+1}(y^u_n))$ for every $j\in \{0,\dots,nk_0\}$. Moreover $y^u_n\in \cW^{uu}_g(x,2L_0)$. 
		
		By the same reasons than above, if $v\in \cC^u(g^{ik_0}(y^u_n))$ we have 
		$$\norm{Dg^{jk_0}_{g^{ik_0}(y^u_n)}(v)}\geq (\lambda_2)^j\norm{v} \ \ \text{for every} \ \ 0\leq i+j \leq n.$$
		Finally, if we take $y \in \cW^{uu}_g(x,2L_0)$ as an accumulation point of $\{y^u_n\}_{n\in \N}$ we obtain that there is a $d$-center cone  $\cC^u(g^l(y))\subset E^c_g(g^l(y))$ such that $Dg(\cC^u(g^l(y)))\subset \cC^u(g^{l+1}(y))$ for every $l\geq 0$ and $\norm{Dg^{jk_0}(v)}\geq \lambda_2^j\norm{v}$, for every $v\in \cC^u(g^{lk_0}(y))$ and $j,l>0$. 
	\end{proof}
	
	Since the $C^1$-openess of the SH-Saddle property for stable manifolds is completely analogous we get the following corollary.
	
	\begin{cor}
		The SH-Saddle property \eqref{defshf} is $C^1$ open among $\cPH(M)$. 
	\end{cor} 
	We end this section with a key corollary from Theorem \ref{shisopen} that we're going to use in the next sections. First let us say that $D$ is a center disk of dimension $d\leq \textnormal{dim}E^c_f$ if it is a $d$-dimensional embedded disk contained in some center plaque. 
	
	\begin{cor} \label{corotamanodisco}
		
		Let  $f\in \cPH(M)$ be such that its unstable foliation has the SH-Saddle property of index $d$ and let $\lambda>1$, $\delta_1>0$ and $\mathcal{V}$ as in the Theorem \ref{shisopen}. Take $g\in \mathcal{V}$, $x^u \in H^+_{\lambda,d}(g)$ and $D^u$ a center disk of dimension $d$ tangent to $\cC^u(x^u)$. Then there is $N>0$ such that $g^n(D^u)$ contains a center disk of dimension $d$, centered at $g^n(x^u)$ of diameter bigger than $2\delta_1$ for every $n\geq N$. 
		
		Analogously with the stable foliation. 
	\end{cor}
	
	\begin{proof} 
		First recall that $\mathcal{U}_1(f)$ and $\delta_1>0$ come from Equation \eqref{constantdelta1}. Then just notice that if $g\in \mathcal{V}\subseteq \mathcal{U}_1(f)$ and $d(x,x^u)<\delta_1$, then their distance is expanded by $\lambda>1$ for the future in the $\mathcal{C}^u$ direction. In particular their distance in the center leaf is $d(g(x),g(x^u))\geq \lambda d(x,x^u)\geq \lambda\delta_1>\delta_1$. Then, no matter how small is the center disk $D^u$, eventually by induction we obtain a center disk with diameter bigger than $2\delta_1$. 
	\end{proof}

	\section{Derived from Anosov revisited}  \label{sderivedfromanosov}
	
	In this section we are going to present a sufficient condition for derived from Anosov diffeomorphisms to be $C^1$ robustly transitive. We begin by explaining what we mean with derived from Anosov diffeomorphisms.
	
	Take $n\in \N$ with $n\geq 4$ and let $p:\R^n\to \R^n/\Z^n=\T^n$ be the canonical projection. Take $A\in \text{SL}(n,\Z)$ a hyperbolic matrix with a dominated spiltting of the form 
	\begin{equation}\label{matrixA} \R^n=E^{ss}_A\oplus E^{ws}_A\oplus E^{wu}_A\oplus E^{uu}_A
	\end{equation} and denote by $f_A$ to the linear Anosov diffeomorphism induced in the torus $\T^n$, i.e. $f_A \circ p =p\circ A$. By a slightly abuse of notation we are going to note $f_A=A$. In addition to this decomposition, we can group together the two middle bundles and call $E^c_A=E^{ws}_A\oplus E^{wu}_A$ to the center bundle. This way we get a dominated spiltting of the form $$\R^n=E^{ss}_A\oplus E^{c}_A\oplus E^{uu}_A.$$ With this splitting we can think of $A$ as a partially hyperbolic diffeomorphism too. We say that $f:\T^n\to\T^n$ is a \textit{derived from Anosov} diffeomorphism if it is isotopic to $A$. Now we let $\cPH_A(\T^n)$ be the set 
	\begin{equation*}
		\cPH_{A}(\mathbb{T}^n)= \left\{
		f \in \cPH(\mathbb{T}^n): f \simeq A, \ \textnormal{dim}E^{*}_{f}=\textnormal{dim}E^{*}_{A}, \ \text{for} \ *=ss,c,uu 
		\right\}
	\end{equation*} 
	where $f\simeq A$ means the maps are isotopic. Then $\cPH_{A}(\mathbb{T}^n)$ is the set of partially hyperbolic derived from Anosov diffeomorphism, such that the dimensions of the subbundles coincide with the dimensions of the linear subbundles. By the results of \cite{FPS} we know that every $f\in \cPH_{A}(\mathbb{T}^n)$ is dynamically coherent, i.e. the bundles $E^{ss}_f\oplus E^c_f$, $E^c_f\oplus E^{uu}_f$ and $E^c_f$ integrate to invariant foliations denoted by $\cWcs_f$, $\cWcu_f$ and $\cWc_f$ respectively.

	\subsection{Derived from Anosov with SH-Saddle property}
	In this subsection we are going to show that for every linear Anosov $A:\T^n\to \T^n$ as above, there is a derived from Anosov diffeomorphism with the SH-Saddle property with a given index (actually the same index as its linear part). 
	
	
	\begin{lem} \label{DAesSH}
		Let $A\in \textnormal{SL}(n,\Z)$ be a hyperbolic matrix with a dominated splitting as in Equation \eqref{matrixA}. Take $\e>0$ and call $U=B(0,\e) \subset \T^n=\R^n/\Z^n$. Take $f_t:\T^n\to \T^n$ an isotopy such that:
		\begin{enumerate} 
			\item $f_0=A$ and $f_1=f$, \label{remdeff1}
			\item $f_t|_{U^c}=A|_{U^c}$, for every $t\in [0,1]$, \label{remdeff2}
			\item $\text{dim}E^{*}_{f_t}=\text{dim}E^{*}_{A}$, for every $*=ss,c,uu$ and every $t\in [0,1]$. \label{remdeff3}
		\end{enumerate}
		Then, if $\e$ is sufficiently small, $f\in \cPH_A(\T^n)$ has the $(d_1,d_2)$ SH-Saddle property, where $d_1=\textnormal{dim}E^{ws}_A$ and $d_2=\textnormal{dim}E^{wu}_A$.
		
	\end{lem}
	\begin{proof}
		Take $f$ as in the hypotesys. We can assume that $\e$ is small enough in order to send $f_t$ to the quotient $\T^n=\R^n/\Z^n$. It is clear that a diffeomorphism $f$ built this way belongs to $\cPH_A(\T^n)$. By taking an iterate we can suppose that $\norm{Df_{x}|_{E^{uu}_f(x)}}>4$ for every $x\in \T^n$. Now take $0<\epsilon<1/4$. Then for every $x\in \T^n$, there is a point $z^{x}_0 \in \mathcal{W}^{uu}_{f}(x,1)$ such that $\mathcal{W}^{uu}_{f}(z^{x}_0,1/4)\cap U =\emptyset$. Call $\cD_{0}=\overline{\mathcal{W}^{uu}_{f}(z^{x}_0,1/4)}$. In the same way since  $f(\cD_{0})\supseteq\mathcal{W}^{uu}_{f}(f(z^x_0),1)$, we can find a disk $\cD_{1}=\overline{\mathcal{W}^{uu}_f(z^{x}_1,1/4)} \subset f(\cD_0)$ such that $\cD_1\cap U=\emptyset$. Inductively we get a sequence of unstable disks $\{\cD_j\}_{j\geq0}$ such that  $\cD_j\cap U=\emptyset$ for every $j\geq 0$ and $f^{-1}(\cD_j) \subset \cD_{j-1}$ (see Figure \ref{DASH} above). Finally the point $x^u=\bigcap_{j\geq 0}f^{-j}(\cD_j)$ never meets $U$ in the future.  Since $f$ is equal to $A$ outside $U$ we get that the point $x^u$ is hyperbolic for the future, and so the unstable manifold $\cW^{uu}_f$ has SH-Saddle property of index $d_2$. 
		
		\begin{figure}[H]
			\begin{center}
				\includegraphics [width=10cm]{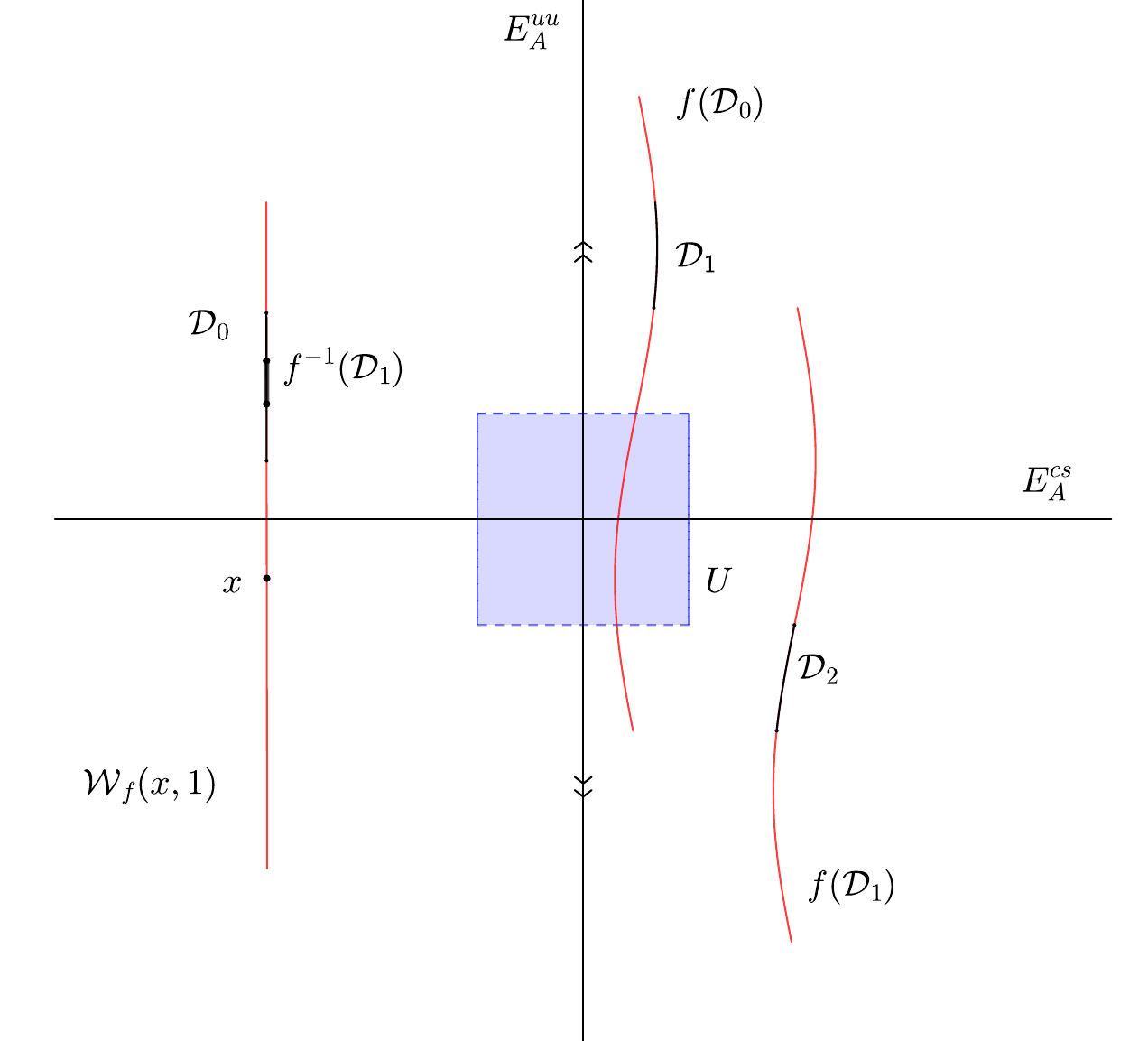}
				\caption{Finding a point whose forward orbit never meets $U$} \label{DASH}
			\end{center}
		\end{figure} 
		\vspace{-0.5cm}
		In the same way we can find a point $x^s$ in every strong stable leaf of large 1, such that the past orbit of $x^s$ never meets $U$. Once again since $f=A$ outside $U$, the same argument as above shows that $\cW^{ss}_f$ has SH-Saddle property of index $d_1$.	
	\end{proof}

	\subsection{A criterion for openess} \label{scopeness}
	
	In this subsection we present a result that we are going to apply for the proof of the main theorem. Roughly speaking it says that given a continuous function between topological spaces of the same dimension, and such that the fibers (preimages of points) of the funcion are small enough, then the image of the funtion must contain an open set. The version we are going to use comes from \cite{LZ} which is an improvement from a result of \cite{BK} (Proposition 3.2). We begin with a few definitions.
	
	\begin{df}
		Suppose $f:X\to Y$ is a continuous function between metric spaces. We say that $y\in Y$ is a \textit{stable value} if there is $\epsilon>0$ such that if $d_{C^0}(f,g)<\epsilon$ then $y\in Im(g)$.
	\end{df}
	
	\begin{rem} \label{remcontieneabierto}
		Let $Y=\mathbb{R}^n$ and suppose that $f:X \to \mathbb{R}^n$ has a stable value $y$, then $Im(f)$ contains an open set. To see this, take $\epsilon>0$ from the definition of stable value, and take a vector $v\in \mathbb{R}^d$ with $\norm{v}<\epsilon$. Then the map $g:X\to \mathbb{R}^d$ defined by $g(x)=f(x)-v$ satisfies $d_{C^0}(f,g)= \norm{v}<\epsilon$. Since $y$ is a stable value, there is a point $x\in X$ such that $g(x)=y$ and this is equivalent to $f(x)=y+v$. Since $v$ was arbitrary we get $B_{\mathbb{R}^n}(y,\epsilon)\subset Im(f)$.
	\end{rem}
	
	\begin{df} \label{dflight}
		Given a continuous function $f:X\to Y$ and $\rho>0$ we say that $f$ is $\rho-$light if for every $y \in Y$ the connected components of $f^{-1}(y)$ have diameter smaller than $\rho$.
	\end{df}
	
	\begin{prop}[Theorem F in \cite{LZ}] \label{propLZadapted} Given $d\in \mathbb{N}$ and $r>0$ there is $\rho=\rho(d,r)>0$ such that every $\rho$-light map $f:[-r,r]^d\to \mathbb{R}^d$ has a stable value. 
	\end{prop}
	
	The version stated in \cite{LZ} is for maps $f:[0,1]^d \to \mathbb{R}^d$ but the proof can be adapted to maps $f:[-r,r]^d\to \mathbb{R}^d$ for a fixed $r>0$. Now combining this proposition and Remark \ref{remcontieneabierto} we have the following corollary.
	
	\begin{cor} \label{corcontieneabierto}
		Fix $d\in \mathbb{N}$ and $r>0$, and take the corresponding $\rho=\rho(d,r)>0$ from Proposition \ref{propLZadapted}. Then the image of every $\rho$-light map $f:[-r,r]^d\to \mathbb{R}^d$ contains an open set. 
	\end{cor}

	\subsection{Robust transitivity for DA diffeomorphisms}
	In this subsection we are going to present a robust transitivity criterion for DA diffeomorphisms. This result will be used in the next section for the proof of Theorem \ref{teoejemixed}. 
	
	Let $A\in \text{SL}(n,\Z)$ be a hyperbolic matrix with a dominated splitting as in Equation \eqref{matrixA}. Take $f\in \cPH_A(\T^n)$ and let $\wt{f}$ be a lift to $\R^n$. By \cite{Fr} we know there exist a continuous and surjective map $H_f:\R^n \to \R^n$ such that $A \circ H_f=H_f\circ \wt{f}$. The map $H_f$ is $\Z^n$-invariant and therefore it induces a continuous and surjective map $h_f:\T^n\to\T^n$ such that $h_f \circ f=A \circ h_f$. Moreover, the map $H_f$ varies continuously with the diffeomorphism $f$ in the $C^0$-topology and the distance $d_{C^0}(H_f ,\textit{Id}_{\wt{M}})=d_{C^0}(h_f,\textit{Id}_M)<\infty$. In particular we have that $d_{C^0}(H_f,\textit{Id}_{\wt{M}})\to 0$ when $f\to A$ in the $C^0$ topology.
	
	Notice that we are making an abuse of notation since the map $H_f$ is determined by $\wt{f}$ instead of $f$. But this is not a problem since given two lifts $\wt{f}_1$ and $\wt{f}_2$ there is an integer vector $v\in \Z^n$ such that $\wt{f}_1-\wt{f}_2=v$ and this implies that $H_{f_2}=H_{f_1}+w$, where $w=-(A-\textit{Id})^{-1}(v)$: 
	\begin{eqnarray*}
		H_{f_2} \circ \wt{f_2}(\wt{x}) &=& H_{f_1}(\wt{f_2}(\wt{x}))+w=H_{f_1}(\wt{f_1}(\wt{x})-v)+w \\
		&=& H_{f_1}\circ \wt{f_1}(\wt{x})-v+w= A \circ H_{f_1}(\wt{x})-v+w \\
		&=& A(H_{f_1}(\wt{x})+w)-Aw +w-v = A \circ H_{f_2}(\wt{x}) -(A-\textit{Id})^{-1}(w)-v \\
		&=& A \circ H_{f_2}(\wt{x}).
	\end{eqnarray*}
	Observe that the matrix $A-\textit{Id}$ is invertible since $A$ is hyperbolic. 
	
	Now given $f\in \cPH_A(\T^n)$ and $\wt{x}\in \R^n$ we are going to call the \textit{fiber} of $\wt{x}\in \R^n$ to the set $H_f^{-1}(H_f(\wt{x}))$. By the previous observation given two lifts $\wt{f}_1$ and $\wt{f}_2$ there is a vector $w\in \R^n$ such that $H_{f_2}=H_{f_1}+w$ and this implies that $$ H_{f_2}^{-1}(H_{f_2}(\wt{x})) = H_{f_1}^{-1}(H_{f_1}(\wt{x}))$$ and the fiber does not depend on the choice of the lift. As a result we can define the function \textit{size of the fiber} $$\Lambda: \cPH_A(\T^n)\times \R^n \to \R_{\geq 0} \ \ \ \text{by} \ \ \  \Lambda(f,\wt{x})=\text{diam}(H_f^{-1}(H_f(\wt{x}))).
	$$ We also note by 
	\begin{equation*} 
		\Lambda(f)=\sup\{\Lambda(f,\wt{x}):\wt{x}\in \R^n\}
	\end{equation*} to the supremum of sizes within all fibers. Since $d_{C^0}(H_f,\textit{Id}_{\R^n})<\infty$ this supremum is always finite and we get a well defined function $\Lambda: \cPH_A(\T^n) \to \R_{\geq 0}$. Notice that $f$ is conjugated to $A$ if and only if $\Lambda(f)=0$, since $H_f$ is always surjective and $\Lambda(f)=0$ is equivalent to injectivity.
	
	It's easy to see that the function $\Lambda$ does not depend continuously on $f$, however we have an upper semicontinuity property as the following lemma shows.
	
	\begin{lem} \label{controlfibras}
		Let $f\in \cPH_A(\mathbb{T}^n)$. Then for every $\e>0$ there exist $\delta>0$ such that:
		if $d_{C^0}(f,g)<\delta$ then $\Lambda(g)< \Lambda(f)+\e$. 
	\end{lem}
	\begin{proof}
		Take $f\in \cPH_A(\T^n)$ and $\e>0$. Suppose by contradiction that the lemma is false. Then for every $k>0$ there is $g_k\in \cPH_A(\T^n)$ with $d_{C^0}(g_k,f)\leq 1/k$, and points $\wt{x_k}, \wt{y_k} \in \R^n$ such that $d(\wt{x_k},\wt{y_k})\geq \Lambda(f)+\e$ and $H_{g_k}(\wt{x_k})=H_{g_k}(\wt{y_k})$. We can assume that $\wt{x_k}\to \wt{x}$ and $\wt{y_k}\to \wt{y}$, and in consequence $d(\wt{x},\wt{y})\geq \Lambda(f)+\e$. 
		Since the map $g\mapsto H_g$ is continuous, for every $\delta>0$ there is $k_0>0$ such that for every $k\geq k_0$, we have $d_{C^0}(H_{g_k},H_f)<\delta$. Then by the triangular inequality we have
		\begin{eqnarray*}
			d(H_f(\wt{x}),H_f(\wt{y})) & \leq & d(H_f(\wt{x}),H_f(\wt{x_k}))+d(H_f(\wt{x_k}),H_{g_k}(\wt{x_k}))+d(H_{g_k}(\wt{x_k}),H_{g_k}(\wt{y_k})) \\
			& + & d(H_{g_k}(\wt{y_k}),H_f(\wt{y_k}))+d(H_f(\wt{y_k}),H_f(\wt{y})) \\
			& \leq & d(H_f(\wt{x}),H_f(\wt{x_k}))+2\delta + d(H_f(\wt{y}),H_f(\wt{y_k})) \to 2\delta
		\end{eqnarray*} and this implies $H_f(\wt{x})=H_f(\wt{y})$, since $\delta$ was arbitrary. As a result, the points $\wt{x}$ and $\wt{y}$ belong to the same fiber which implies $d(\wt{x},\wt{y})\leq \Lambda(f)$. But then we have $\Lambda(f)+\epsilon \leq d(\wt{x},\wt{y}) \leq \Lambda(f)$ which is a contradiction.
	\end{proof}

	Recall that a diffeomorphism $f:M \to M$ is said to be \textit{transitive} if there is $x\in M$ such that $\ol{\cO^+ (f,x)}=M$. Equivalently, $f$ is transitive if for every pair of open sets $U$ and $V$ there is $N\in \mathbb{Z}^+$ such that $f^N(U)\cap V \neq \emptyset$.

	Now we are ready to prove the main theorem of this section.
	
	\begin{thm}[Robust transitivity criterion] \label{teorobtran}
		Let $A\in \textnormal{SL}(n,\Z)$ be a hyperbolic matrix with a dominated splitting as in Equation \eqref{matrixA}. Take $f \in \cPH_A(\T^n)$ with $(d_1,d_2)$ SH-Saddle property where $d_1=\textnormal{dim}E^{ws}_A$ and $d_2=\textnormal{dim}E^{wu}_A$. Then there is $\rho=\rho(f)>0$ such that if $\Lambda(f)<\rho$ then $f$ is $C^1$ robustly transitive. In fact $C^1$ robustly topologically mixing.
	\end{thm}
	
	\begin{proof} 
		Take $f \in \cPH_A(\T^n)$ with $(d_1,d_2)$ SH-Saddle property where $d_1=\text{dim}E^{ws}_A$ and $d_2=\text{dim}E^{wu}_A$.	Let $\mathcal{V}$, $\lambda>1$ and $\delta_1>0$ be as in Theorem \ref{shisopen}.
		
		Let us define the following constants:
		\begin{eqnarray*}
			\rho_s &=& \rho(\text{dim}(E^{ss}_A\oplus E^{ws}_A),\delta_1) \\
			\rho_u &=& \rho(\text{dim}(E^{wu}_A\oplus E^{uu}_A),\delta_1) \\
			\rho &=& \min \{\rho_s,\rho_u\}
		\end{eqnarray*} where $\rho(*,\delta_1)$ are given by Proposition \ref{propLZadapted} for $*=\text{dim}(E^{ss}_A\oplus E^{ws}_A), \text{dim}(E^{wu}_A\oplus E^{uu}_A)$. We claim that the theorem holds for  this $\rho>0$ and for proving this we are going to find a $C^1$-neighbourhood $\cU(f)$ of $f$ such that every $g\in \cU(f)$ is transitive.  
		
		First notice that since $\Lambda(f)<\rho(f)$, by Lemma \ref{controlfibras} applied to $\e= \rho(f) - \Lambda(f)>0$, we know there is $\delta_0>0$ such that if $d_{C^0}(f,g)<\delta_0$ then $\Lambda(g)<\Lambda(f)+\e=\rho(f)$. 
		
		Now take $\cU(f)=\mathcal{V}\cap \{g\in \cPH_A(\T^n): d_{C^0}(f,g)<\delta_0\}$. We claim that every $g\in \cU(f)$ is transitive (in fact topologically mixing). In order to get transitivity, we have to prove that for any two open sets $U_1, U_2 \subset \T^n$ there is $k \in \Z^+$ such that $g^k(U_1)\cap U_2 \neq \emptyset$. 
		
		Take two points $x_1\in U_1$ and $x_2\in U_2$, and let $k_1 \in \N$ be such that $g^{-k_1}(U_1)\supset \cW^{ss}_g(g^{-k_1}(x_1),L)$ and $g^{k_1}(U_2)\supset \cW^{uu}_g(g^{k_1}(x_2),L)$. Take $x^s \in H^-_{\lambda,d_1}(g) \cap \cW^{ss}_g(g^{-k_1}(x_1),L)$ and $x^u \in H^+_{\lambda,d_2}(g) \cap \cW^{uu}_g(g^{k_1}(x_2),L)$ given by $(d_1,d_2)$ SH-Saddle property. Now take $D^s \subset \cW^c_g(x^s)$ a center disk of dimension $d_1$ tangent to $\cC^s(x^s)$ and $D^u \subset \cW^c_g(x^u)$ a center disk of dimension $d_2$ tangent to $\cC^u(x^u)$. We can take $D^s, D^u$ small enough such that $D^s \subset g^{-k_1}(U_1)$ and $D^u \subset g^{k_1}(U_2)$. Recall that $\cC^s$ and $\cC^u$ are the cones invariant for the past and the future respectively given by SH-Saddle property. Moreover, $\cC^s$ and $\cC^u$ uniformly expand vectors for the past and the future respectively. 		
		\begin{figure}[H]
			\begin{center}
				\includegraphics [width=15cm]{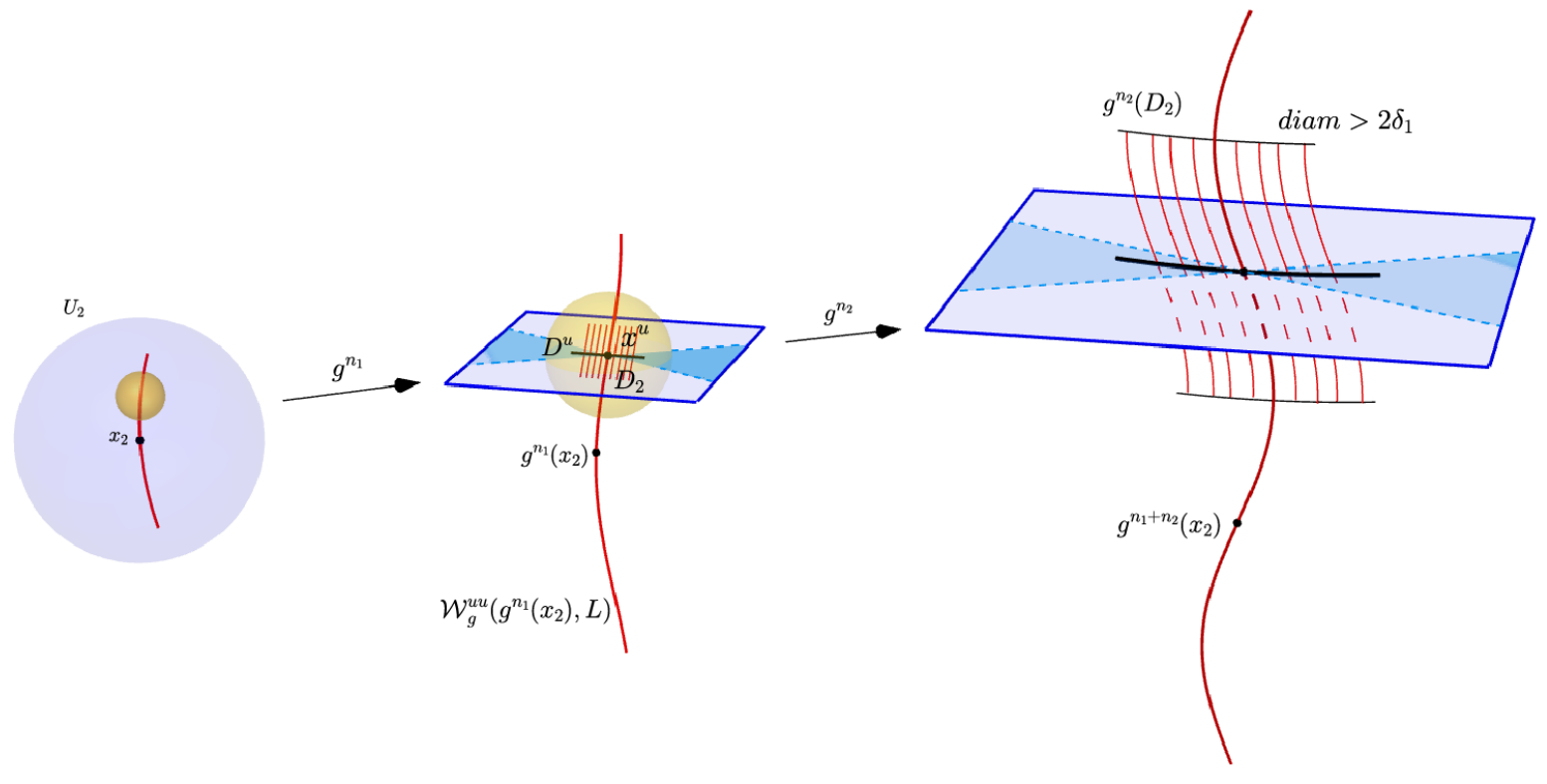}
				\caption{Obtaining a disk of diameter bigger than $2\delta_1$}
			\end{center}
		\end{figure}
		\vspace{-0.5cm}
		Now take $D_1=\cup_{x\in D^s} \cW^{ss}_g(x,l)$ and $D_2=\cup_{x\in D^u}\cW^{uu}_g(x,l)$. We can choose $l>0$ small enough such that $D_1 \subset g^{-k_1}(U_1)$ and $D_2 \subset g^{k_1}(U_2)$. 
		Notice that $D_1$ is a disk of dimension equal to $\text{dim}(E^{ss}_A\oplus E^{ws}_A)$ and $D_2$ is a disk of dimension equal to $\text{dim}(E^{wu}_A\oplus E^{uu}_A)$. Now by Corollary \ref{corotamanodisco} there is $k_2\in \N$ such that $g^{-k}(D^s)$ contains a disk of diameter bigger than $2\delta_1$ and $g^k(D^u)$ contains a disk of diameter bigger than $2\delta_1$ for every $k\geq k_2$. 
		
		Now let $\Pi^{s}:\R^n\to E^{ss}_A\oplus E^{ws}_A$  and $\Pi^{u}:\R^n\to E^{wu}_A\oplus E^{uu}_A$ be the orthogonal projections. Then the idea is to use Corollary \ref{corcontieneabierto} applied to the functions $\Pi^s \circ H_g$ and $\Pi^u \circ H_g$ to conclude that for every $k\geq k_2$ the images of the sets $g^{-k}(D_1)$ and $g^{k}(D_2)$ by $H_g$ contain topological disks of complementary dimensions and with the appropriate inclination. Then the hyperbolicity of the matrix $A$ will do the mixing, and we can translate this mixing of $A$ to the diffeomorphism $g$. 
		
		Observe that $g \in \cU(f)$ which implies that $\Lambda(g)<\rho$ and in particular we have that $H_g$ is $\rho$-light (see Definition \ref{dflight}). Moreover we claim the following.
		
		\begin{claim}
			The function $\Pi^{s}\circ H_g$ is $\rho$-light when restricted to $\wt{g}^{-k}(\wt{D_1})$ and the function $\Pi^u\circ H_g$ is $\rho$-light when restricted to $\wt{g}^{k}(\wt{D_2})$, for every $k\geq k_2$. 
		\end{claim}	
		\begin{proof}
			We are going to see the case $\Pi^{s}\circ H_g$ since the other one is symmetric. Now notice that $\wt{g}^{-k}(\wt{D_1})$ contains a disk of size bigger than $2\delta_1$ for every $k \geq k_2$ and the disk $\wt{g}^{-k}(\wt{D_1})$ is tangent to a cone $\cC^s$ which is uniformly expanding for the past. Thus by the semiconjugacy relation $H_g \circ \wt{g} =A \circ H_g$ we know that $H_g(\wt{D_1})$ can not intersect $E^{wu}_A\oplus E^{uu}_A$ more than once, otherwise there would be different points in $\wt{D_1}$ such that their distance by past iterates of $\wt{g}$ goes to zero, and this is impossible since the cones $\cC^s$ are expanding for the past. In consequence the fibers of $\Pi^s \circ H_g $ have the same size of the fibers of $H_g$, and so $\Pi^s \circ H_g$ is $\rho$-light restricted to $\wt{g}^{-k}(\wt{D_1})$ for every $k\geq k_2$. 
		\end{proof}
		
		To sum up, we have a continuous map $\Pi^s \circ H_g: \wt{g}^{-k}(\wt{D_1}) \to E^{ss}_A \oplus E^{ws}_A \simeq \R^{\text{dim}(E^{ss}_A\oplus E^{ws}_A)}$ such that its domain $\wt{g}^{-k}(\wt{D_1})$ contains a disk $[-\delta_1,\delta_1]^{\text{dim}(E^{ss}_A\oplus E^{ws}_A)}$ and by our choice of $\rho$ we have that $\rho \leq \rho(\text{dim}(E^{ss}_A\oplus E^{ws}_A),\delta_1)$. Then just notice that we are in hypothesys of Corollary \ref{corcontieneabierto} and therefore $\Pi^s\circ H_g(\wt{g}^{-k}(\wt{D_1})) \subset E^{ss}_A\oplus E^{ws}_A$ contains an open set for every $k\geq k_2$. The same argument shows that $\Pi^u \circ H_g(\wt{g}^{k}(\wt{D_2})) \subset E^{wu}_A\oplus E^{uu}_A$ contains an open set. 
		
		Since $A$ is a hyperbolic matrix and the topological disks have complementary dimensions and with the right inclination, we know there is $k_3 \in \N$ such that for every $k \geq k_3$ we have that $A^{k}(H_g(\wt{g}^{k_2}(\wt{D_2})))\cap (H_g(\wt{g}^{-k_2}(\wt{D_1}))+V_k) \neq \emptyset$ for some $V_k \in \Z^n$. This implies that $H_g\circ \wt{g}^k(\wt{g}^{k_2}(\wt{D_2}))\cap (H_g(\wt{g}^{-k_2}(\wt{D_1}))+V_k) \neq \emptyset$. Since $H_g$ is at bounded distance to the identity, we know that there is $k_4 \in \N$ such that for every $k \geq k_4$, we have $\wt{g}^k(\wt{g}^{k_2}(\wt{D_2}))\cap (\wt{g}^{-k_2}(\wt{D_1})+V_k) \neq \emptyset$. Then since $p:\R^n \to \T^n$ satisfies $p \circ \wt{g}=g \circ p$ we have that: 
		$$\emptyset \neq g^k(g^{k_2}(D_2))\cap g^{-k_2}(D_1) \subset g^{k+k_1+k_2}(U_2)\cap g^{-k_1-k_2}(U_1)$$ for every $k \geq k_4$ and this is equivalent to 
		$$ \emptyset \neq g^{k+2k_1+2k_2}(U_2)\cap U_1, \ \ \text{for every} \ k \geq k_4.$$ 
		Finally if we take $N=k_4+2(k_1+k_2)$ we have that $g^k(U_2)\cap U_1\neq \emptyset$ for every $k \geq N$ proving that $g$ is topologically mixing. This ends the proof. 
	\end{proof}
	
	\begin{cor} \label{coroSHtransitivo}
		Let $A\in \textnormal{SL}(n,\Z)$ be a hyperbolic matrix with a dominated splitting as in Equation \eqref{matrixA} and let $f \in \cPH_A(\T^n)$ with $(d_1,d_2)$ SH-Saddle property where $d_1=\textnormal{dim}E^{ws}_A$ and $d_2=\textnormal{dim}E^{wu}_A$. If $\Lambda(f)=0$ then $f$ is $C^1$ robustly transitive. 
	\end{cor}
	\begin{proof}
		Since $f$ has SH-Saddle property, we know that $\rho(f)>0$. Then we trivially have $\Lambda(f)=0<\rho(f)$ and we conclude by Theorem \ref{teorobtran}. 
	\end{proof}

	
	\section{Proof of Theorem \ref{teoejemixed}} \label{ssymplexample}
	In this section we prove Theorem \ref{teoejemixed}, i.e. we construct $C^1$ robustly transitive diffeomorphisms with any center dimension and with as many different behaviours on center leaves. Along the proof we are going to perform different isotopies depending on the type of local behaviour we are looking for, i.e. increase or decrease the index of a fixed point, mix two subbundles, etc. For the construction of the local isotopies, we are going to use an auxiliary function that will be used many times.

	\begin{lem} \label{lemauxfunction}
		Let $b>0$. Then for every $\epsilon>0$ there is a function $\beta: \R^+ \cup \{0\}\to \R$ such that:
		\begin{enumerate}
			\item $\beta$ is $C^{\infty}$ and decreasing. 
			\item $\beta$ is supported in $[0,\epsilon]$. 
			\item $\beta(0)=b$.
		\end{enumerate}
	\end{lem} 
	\begin{proof}
		First take a $C^{\infty}$ function $\psi$ supported in $[0,\epsilon]$ such that $\int_0^{\epsilon}\psi(t)dt=b$. Now just take $\beta$ as: $\beta(t)=b-\int_0^t \psi(s)ds$. This function satisfies the lemma.
	\end{proof}

	\subsection{Expansive DA diffeomorphisms} \label{ssexpda}
	
	In this subsection we are going to build expansive DA diffeomorphisms which are partially hyperbolic, but not Anosov. 
	
	\begin{df}[Expansive homeomorphisms] \label{dfexpansive}
		Let $f:M\to M$ be a homeomorphism on a metric space $(M,d)$. We say that $f$ is expansive if there is $\gamma>0$ such that the following holds: if $x,y\in M$ are such that $d(f^n(x),f^n(y)) \leq \gamma$ for every $n \in \Z$, then $x=y$. We call $\gamma$ the expansivity constant.
	\end{df}
	
	In short, expansivity means that two different points in $M$ are $\gamma$ separated eventually in time. Hyperbolic diffeomorphisms are clearly expansive. Then the idea to construct a non-Anosov expansive diffeomorphism, is to start with a linear Anosov and introduce an isotopy in a small neighbourhood of a fixed point $p$, in order to make the derivative of $p$ restricted to some center subbundle equal to the identity, and keeping the dynamics hyperbolic in the rest of the manifold. We remark that dealing with expansivity is quite delicate, so the construction has to be made with some care. In order to prove expansivity we are going to use the following criterium due to J. Lewowicz.
	
	\begin{prop}[\cite{L}] \label{LemaLew}
		Let $f:M\to M$ be a homeomorphism on a compact metric space $(M,d)$. Suppose there is $\beta>0$ and a continuous funcion $V:\{ (x,y)\in M\times M: d(x,y)\leq \beta \} \to \R$ such that $V(x,x)=0$ for every $x\in M$. Let $\Delta V(x,y)=V(f(x),f(y))-V(x,y)$. If there is $\gamma>0$ such that $\Delta V(x,y)>0$ if $0<d(x,y)\leq \gamma$, then $f$ is $\gamma$-expansive. 
		
		The function $V$ is called a Lyapunov function.
	\end{prop}
	
	
	\subsubsection{Two dimensional center bundle}
	
	We begin with the case of a partially hyperbolic diffeomorphism with two dimensional center bundle. We are going to focus on the case where the center bundle $E^c_A$ behaves hyperbolic, since the case where it is completely contractive or expanding will be included later as a special case of the higher dimensional center case. In Lemma \ref{lemmixsub} we construct the example in $\R^2$ and then in Lemma \ref{ejendim4} we insert the example as the restriction of a partially hyperbolic diffeomorphism to one of its center leaves. 
	
	
	\begin{lem} \label{lemmixsub}
		Let $A\in \textnormal{M}_{2\times 2}(\R)$ be a hyperbolic matrix, such that $A(x,y)=(\lambda x, \mu y)$ with $0<\lambda<1<\mu$, and take $\e>0$. Then there exists a diffeomorphism $g:\R^2\to \R^2$ such that:
		
		
		\begin{itemize}
			\item $g(x,y)=A(x,y)$, for every $(x,y)\in B(0,\e)^c$.
			\item $Dg_{(0,0)}=\text{Id}$.
			\item There exists a Lyapunov function $V: \R^2 \times \R^2 \to \R$ such that $\Delta V((x_1,y_1),(x_2,y_2))>0$ for every pair of points $(x_1,y_1),(x_2,y_2) \in \R^2 \times \R^2$.
			
		\end{itemize}
	\end{lem}
	
	
	\begin{proof} 
		Take a hyperbolic matrix $A \in \textnormal{M}_{2 \times 2}(\R)$ as in the hypothesis and let $\e>0$. Let $\beta_{\lambda}$ be the function given by Lemma \ref{lemauxfunction} for $b=1-\lambda$ and its corresponding function $\psi_{\lambda}$, and let $\beta_{\mu}$ be the function given by Lemma \ref{lemauxfunction} for $b=\mu-1$ and its corresponding function $\psi_{\mu}$. Now we define the function $g:\R^2\to \R^2$ by the equation:
		\begin{equation} \label{eqdefg} g(x,y)=(\lambda x,\mu y) + (\beta_{\lambda}(r)x,-\beta_{\mu}(r)y) 
		\end{equation} where $r=x^2+y^2$. Notice that if $r\geq \epsilon$ then $g=A$. In particular $Dg_{(x,y)}=A$ for every $(x,y)$ such that $x^2+y^2\geq \epsilon$. In case $r<\epsilon$ the differential is:
		\begin{equation*} Dg_{(x,y)}=
			\begin{bmatrix}
				\lambda +\beta_{\lambda}(r) +2x^2 \beta_{\lambda}'(r) & 2xy\beta_{\lambda}'(r) \\
				-2xy\beta_{\mu}'(r) & \mu - \beta_{\mu}(r)-2y^2 \beta_{\mu}'(r) 
			\end{bmatrix}. 
		\end{equation*}
		In particular we have that $Dg_{(0,0)}=\textit{Id}$ and therefore $g$ is not hyperbolic. In case $r>0$ we have that
		$$ \lambda+ \beta_{\lambda}(r)+2x^2\beta_{\lambda}'(r)< \lambda + \beta_{\lambda}(0)=1 
		$$ and
		$$ \mu -\beta_{\mu}(r)-2y^2\beta_{\mu}'(r)> \mu -\beta_{\mu}(0)=1. 
		$$ 
		Now take the family of cones in $\R^2$
		$$ \cC^u(x,y)=\{ (a,b)\in \R^2: |a|\leq |b| \}.
		$$ We claim that this familiy of cones is $Dg$-invariant. Therefore, we have to prove that if $(a,b)\in \cC^u$ then $Dg(a,b)=(a_1,b_1)\in \cC^u$, and this occurs if and only if $|a_1|\leq |b_1|$. By the equations above we have that:
		\begin{eqnarray*}
			a_1 &=& a \left( \lambda+ \beta_{\lambda}(r)+2x^2\beta_{\lambda}'(r) \right) + b \left( 2xy\beta_{\lambda}'(r)  \right) \\
			b_1 &=& a \left( 	-2xy\beta_{\mu}'(r) \right) + b \left( \mu - \beta_{\mu}(r)-2y^2 \beta_{\mu}'(r) \right)
		\end{eqnarray*}
		Notice that if $r>\epsilon$ the cones are $Dg$-invariant since $g=A$. On the other hand, for points close to zero ($r$ small) this is not so clear. Therefore we have to take a little more care with the choice of the functions $\beta_{\lambda}$ and $\beta_{\mu}$, in particular with the functions $\psi_{\lambda}$ and $\psi_{\mu}$. 
		
		In order to find a suitable $C^{\infty}$ function $\psi_{\lambda}$, we are going to approximate it by two $C^0$ functions $\psi_{\lambda}^1$ and $\psi_{\lambda}^0$ by above and below, which implies that $\beta_{\lambda}$ will be approximated by below and above by their corresponding $C^1$ functions $\beta_{\lambda}^1$ and $\beta_{\lambda}^0$ given by Lemma \ref{lemauxfunction}. Then we use the density of $C^{\infty}$ functions to conclude. 
		
		Let $\rho>0$ be sufficienly small such that $3\rho<\epsilon$ and $2\rho^2<\text{min}\{1-\lambda, \mu-1\}$, and let $m\in \R^+$ be such that $2m\rho^2 <1-\lambda<2(m+1)\rho^2$. Now for $i=0,1$ we define the function $\psi_{\lambda}^i$ in the following way (see Figure \ref{bump1}):
		
		$$\psi_{\lambda}^i(r)= \left\{ \begin{array}{lll}
			(m+i)r &   \textit{if}  & r \in [0,\rho]  \\
			\\ (m+i)\rho &  \textit{if} & r \in (\rho, 2\rho ) \\
			\\ (m+i)(-r+3\rho)  &  \textit{if}  & r \in [2\rho,3\rho] 
		\end{array}
		\right.$$
		
		\begin{figure}[H]
			\begin{center}
				\includegraphics [width=9cm]{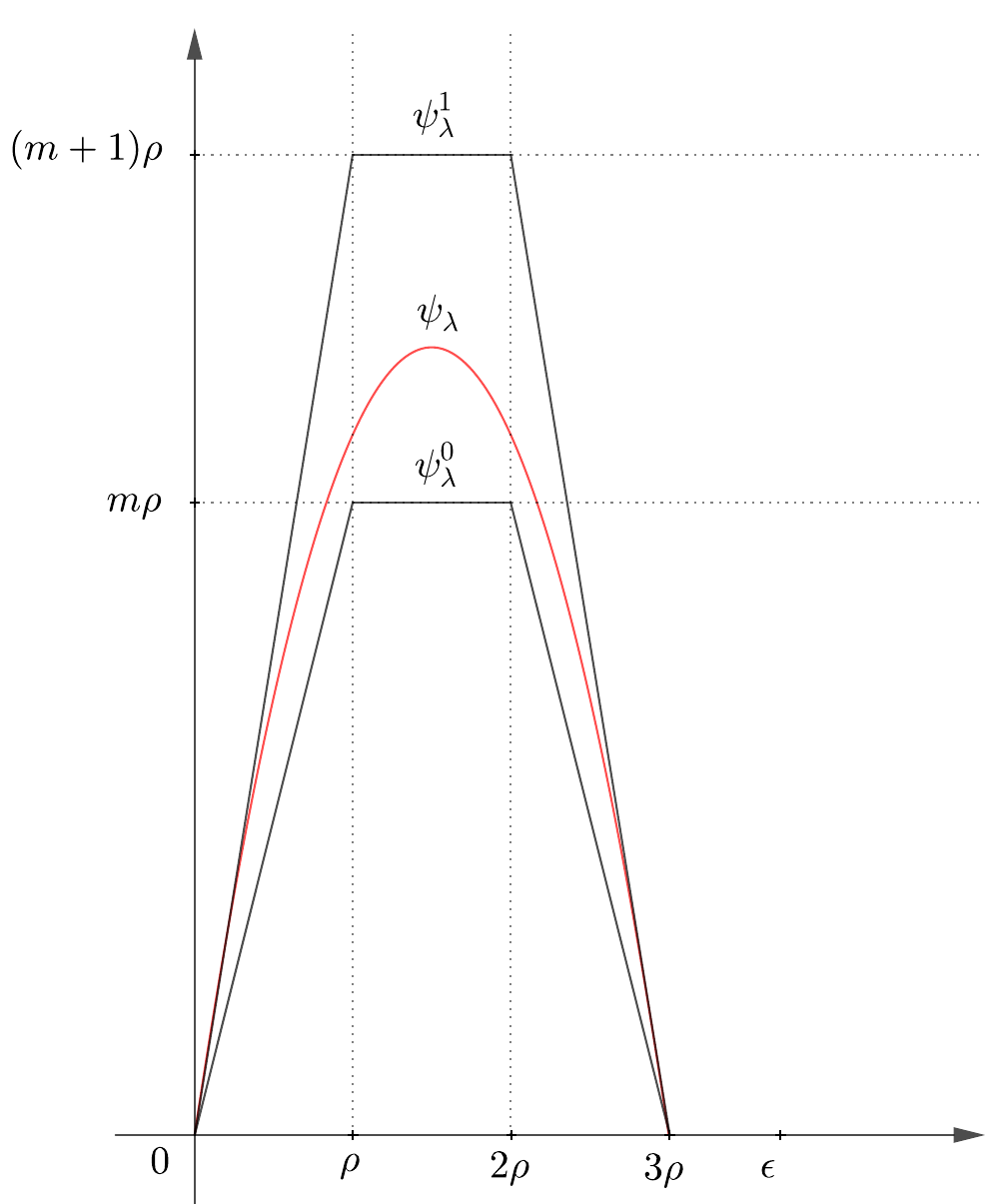}
				\caption{Bump functions $\psi_{\lambda}^0, \psi_{\lambda}^1$ and $\psi_{\lambda}$} \label{bump1}
			\end{center}
		\end{figure}
		Notice that: $$\int_0^{\epsilon}\psi_{\lambda}^0(s)ds=2m\rho^2 < 1-\lambda<2(m+1)\rho^2=\int_0^{\epsilon}\psi_{\lambda}^1(s)ds.$$ 
		Since $C^{\infty}$ functions are $C^0$-dense, we can find a $C^{\infty}$ function $\psi_{\lambda}$ such that 
		\begin{enumerate}
			\item supp$(\psi_{\lambda})\subseteq [0,\epsilon]$.
			\item $\psi_{\lambda}^0(s) \leq \psi_{\lambda}(s) \leq \psi_{\lambda}^1(s)$, for every $s\in[0,\epsilon]$.
			\item $\int_0^{\epsilon}\psi_{\lambda}(s)ds=1-\lambda$.
		\end{enumerate} 
		Notice that point (2) implies that $\beta_{\lambda}^1(r)\leq \beta_{\lambda}(r)\leq \beta_{\lambda}^0(r)$, for every $r\in [0,\epsilon]$. 
		\begin{rem}
			The function $\psi_{\lambda}$ above (Figure \ref{bump1}) is not $C^{\infty}$ at the points $0$ and $3\rho$. However, this can be fixed as follows. First we can defined $\psi_{\lambda}$ in the interval $(-\epsilon,0)$ by symmetry using the equation $\psi_{\lambda}(t)=-\psi_{\lambda}(-t)$. Then $\psi_{\lambda}$ is $C^{\infty}$ at $0$, and $\beta_{\lambda}$ become $C^{\infty}$ at $0$ (notice that 	$\beta_{\lambda}$ is obtained by integrating $\psi_{\lambda}$ only at the non-negative real line, so this extension to negative numbers is not a problem). To deal with the point $3\rho$, we only have to take a local perturbation of $\psi_{\lambda}$ at this point to have the desire properties. Since $3\rho$ is far enough from 0, this can be done easily.  
		\end{rem}	
		
		In the same way we can find a suitable $C^{\infty}$ function $\psi_{\mu}$ by approximating it by two $C^0$ functions $\psi_{\mu}^1$ and $\psi_{\mu}^0$ by above and below, which implies that $\beta_{\mu}$ will be approximated by below and above by their corresponding $C^1$ functions $\beta_{\mu}^1$ and $\beta_{\mu}^0$ according to Lemma \ref{lemauxfunction}. 
		
		Recall that $2\rho^2<\mu-1$, then let $n\in \R^+$ be such that $2n\rho^2<\mu-1<2(n+1)\rho^2$. Now for $i=0,1$ we define the function $\psi_{\mu}^i$ in the following way:
		$$\psi_{\mu}^i(r)= \left\{ \begin{array}{lll}
			(n+i)r &   \textit{if}  & r \in [0,\rho]  \\
			\\  (n+i)\rho &  \textit{if} & r \in (\rho, 2\rho) \\
			\\ (n+i)(-r+3\rho)  &  \textit{if}  & r \in [2\rho,3\rho] 
		\end{array}
		\right.$$
		In the same way as above we observe that: $$\int_0^{\epsilon}\psi_{\mu}^0(s)ds=2n\rho^2<\mu-1<2(n+1)\rho^2=\int_0^{\epsilon}\psi_{\mu}^1(s)ds.$$Since $C^{\infty}$ functions are $C^0$-dense, we can find a $C^{\infty}$ function $\psi_{\mu}$ such that: 
		\begin{enumerate}
			\item supp$(\psi_{\mu})\subseteq [0,\epsilon]$.
			\item $\psi_{\mu}^0(s) \leq \psi_{\mu}(s) \leq \psi_{\mu}^1(s)$, for every $s\in[0,\epsilon]$.
			\item $\int_0^{\epsilon}\psi_{\mu}(s)ds=\mu-1$.
		\end{enumerate} 
		Notice that point (2) implies that $\beta_{\mu}^1(r)\leq \beta_{\mu}(r)\leq \beta_{\mu}^0(r)$, for every $r\in [0,\epsilon]$. 
		
		Recall that we want to prove that the cones $\mathcal{C}^u$ are $Dg$-invariant. We are going to obtained the invariance of the cones for $C^1$ close diffeomorphisms $g^0$ and $g^1$ defined (in the same way as $g$) by Equation \eqref{eqdefg} but using the functions $\beta_{\lambda}^0$, $\beta_{\mu}^0$ and $\beta_{\lambda}^1$, $\beta_{\mu}^1$ respectively, instead of $\beta_{\lambda}$ and $\beta_{\mu}$. Notice that for $i=0,1$ the $C^1$ distance between $g$ and $g^i$ can be taken as small as we want.
		
		Now for $i=0,1$, in order to see the cones $\mathcal{C}^u$ are $Dg^i$-invariant, we have to prove that if $(a,b)\in \cC^u$ then $Dg^i(a,b)=(a_1,b_1)\in \cC^u$, and this happens if and only if $|a_1|\leq |b_1|$. We will obtain this inequality by studying different cases depending on $r$ which is the square of the distance of the point $(x,y)$ to the origin.  
		\newline
		\newline 
		\textbf{Case 1: $r\in [0,\rho]$.} \newline By our definitions above we have:
		$$\beta_{\lambda}^i(r)=1-\lambda -(m+i)\frac{r^2}{2} \ \ \text{and} \ \ \ (\beta_{\lambda}^i)'(r)=-(m+i)r.
		$$
		Then we have that:
		$$\lambda+ \beta_{\lambda}^i(r)+2x^2(\beta_{\lambda}^i) '(r)=1-(m+i)\frac{r^2}{2}-2x^2 (m+i)r
		$$
		and in consequence we obtain that $a_1$ is equal to:
		$$a_1 = a \left( 1-(m+i)\frac{r^2}{2}-2x^2 (m+i)r \right) + b \left(-2xy(m+i)r \right). 
		$$
		Since $|a|\leq|b|$, by taking absolute value and applying triangular inequality we get:
		\begin{eqnarray*}
			|a_1|&\leq &|b|\left( 1-(m+i)\frac{r^2}{2}-2x^2 (m+i)r \right) + |b|\left( 2|xy|(m+i)r \right) \\
			&=& |b|\left(1 -(m+i)\frac{r}{2}\left( r+4x^2 -4|xy| \right)  \right)  \leq |b| 
		\end{eqnarray*}
		where the last inequality holds as long as $r+4x^2-4|xy|=5x^2+y^2-4|xy|>0$. We claim that this is always the case: if $|xy|=xy$ we have to show that $5x^2+y^2-4xy>0$. By solving the second degree equation in $y$ we obtain that 
		$$y=\frac{4x\pm \sqrt{16x^2-20x^2}}{2} 
		$$ and this has no real roots. Since for $y=0$ we have $5x^2\geq 0$, we obtain the desire inequality. The case where $|xy|=-xy$ is completely the same since the discriminant in the equation above is the same. We conclude that $r+4x^2-4|xy|\geq0$, and moreover $|a_1|<|b|$ if $r>0$.
		
		In the same way we have that:
		$$\beta_{\mu}^i(r)=\mu-1 -\left( n+i\right)\frac{r^2}{2} \ \ \text{and} \ \ \ (\beta_{\mu}^i)'(r)=-\left(n+i\right)r
		$$ and then 
		$$\mu -\beta_{\mu}^i(r)-2y^2(\beta_{\mu}^i)'(r)=1+\left(n+i \right)\frac{r^2}{2}+2y^2\left( n+i \right)r.
		$$
		We then have that $b_1$ is equal to
		$$b_1=a \left(2xy\left( n+i \right)r\right)+ b\left(  1+\left(n+i \right)\frac{r^2}{2}+2y^2\left( n+i \right)r   \right).
		$$ Since $|a|\leq|b|$, taking absolute value and by the triangular inequality we have:
		\begin{eqnarray*}
			|b_1|&\geq& |b|\left( 1+\left(n+i \right)\frac{r^2}{2}+2y^2\left( n+i \right)r \right)-|b| \left(2|xy|\left( n+i \right)r\right) \\
			&=& |b|\left( 1+\frac{r}{2}\left(n+i \right)\left(  r+4y^2-4|xy|\right)\right) \geq |b|
		\end{eqnarray*}
		where the last inequality holds as long as: $r+4y^2-4|xy|=x^2+5y^2-4|xy|\geq 0$. This is exactly the same equation we solve above, and thus we conclude that $|b_1|\geq |b|$ and moreover, $|b_1|>|b|$ if $r>0$. Then we conclude that $|a_1|\leq |b|\leq|b_1|$.
		\newline
		\newline
		\textbf{Case 2: $r\in [\rho,2\rho]$.} \newline In this case we obtain that $\beta_{\lambda}^i$  verifies:
		$$\beta_{\lambda}^i(r)=1-\lambda-(m+i)\frac{\rho^2}{2} -(m+i)\rho(r-\rho) \ \ \text{and} \ \ \ (\beta_{\lambda}^i)'(r)=-(m+i)\rho.
		$$
		Then we have that:
		$$\lambda+ \beta_{\lambda}^i(r)+2x^2(\beta_{\lambda}^i)'(r)=1-(m+i)\frac{\rho^2}{2} -(m+i)\rho(r-\rho) -2x^2(m+i)\rho
		$$
		therefore $a_1$ is equal to:
		$$a_1 = a \left( 1-(m+i)\frac{\rho^2}{2} -(m+i)\rho(r-\rho) -2x^2(m+i)\rho  \right) + b \left( -2xy(m+i)\rho  \right). 
		$$
		Since $|a|\leq|b|$, taking absolute value and applying triangular inequality we get:
		\begin{eqnarray*}
			|a_1|&\leq & |b| \left( 1-(m+i)\frac{\rho^2}{2} -(m+i)\rho(r-\rho) -2x^2(m+i)\rho \right)\\ &+& |b| \left( 2|xy|(m+i)\rho \right) \\
			&=& |b|\left( 1-(m+i)\frac{\rho^2}{2} -(m+i)\rho(r-\rho) -2x^2(m+i)\rho + 2|xy|(m+i)\rho \right) \\	
			&=& |b|\left(1-\left(\frac{(m+i)\rho}{2}\right)\left( 6x^2+2y^2-4|xy|-\rho \right) \right) \leq |b|
		\end{eqnarray*}
		where the last inequality holds as long as $6x^2+2y^2-4|xy|-\rho \geq 0$. Notice that 
		$$6x^2+2y^2-4|xy|-\rho=(x^2+y^2-\rho)+(5x^2+y^2-4|xy|).$$
		The first term in the right expression is greater or equal to zero since $\rho \leq r$. The second term $5x^2+y^2-4|xy|$ is exactly the same equation we solve in case 1. 
		
		In the same way we have that:
		$$\beta_{\mu}^i(r)=\mu-1-\frac{\rho^2}{2}\left(n+i \right)-(r-\rho)\rho\left(n+i\right) \ \ \text{and} \ \ \ (\beta_{\mu}^i)'(r)=-\left(n+i \right)\rho
		$$
		and then we obtain:
		$$\mu-\beta_{\mu}^i(r)-2y^2(\beta_{\mu}^i)'(r)=1+\frac{\rho^2}{2}\left(n+i\right)+(r-\rho)\rho\left( n+i \right)+2y^2\left( n+i \right)\rho.
		$$
		We then have that $b_1$ is equal to
		$$b_1=a \left(2xy\left( n+i \right) \rho\right)+ b\left(1+\frac{\rho^2}{2}\left(n+i\right)+(r-\rho)\rho\left( n+i \right)+2y^2\left( n+i \right)\rho \right).
		$$ Since $|a|\leq|b|$, taking absolute value and by the triangular inequality we have:
		\begin{eqnarray*}
			|b_1|&\geq& |b|\left(1+\frac{\rho^2}{2}\left(n+i\right)+(r-\rho)\rho\left( n+i \right)+2y^2\left( n+i \right)\rho\right)\\ &-& |b| \left( 2|xy|\left( n+i \right) \rho \right) \\
			&=& |b|\left( 1+ \frac{\rho^2}{2}\left( n+i\right)+ (r-\rho)\rho(n+i)+2y^2(n+i)\rho-2|xy|(n+i)\rho \right) \\
			&=& |b|\left( 1+\left(\frac{(n+i)\rho}{2}\right)\left(2x^2+6y^2-\rho-4|xy|	\right)\right) \geq |b|
		\end{eqnarray*}
		by the same estimates than above. Moreover we have that $|b_1|>|b|$ if $r>0$. Then we conclude that $|a_1|\leq|b| \leq |b_1|$.
		\newline
		\newline
		\textbf{Case 3: $r\in [2\rho,3\rho]$} \newline This is the simplest case, since we are far enough to zero and so we omit the calculations.

		\vspace{0.3cm}
		To sum up we have proved that for $i=0,1$ the cones $\mathcal{C}^u$ are $Dg^i$-invariant. In fact we have proved that if $(a,b)\in \cC^u$ then $Dg^i(a,b)=(a_1,b_1)$ verifies $|a_1|\leq |b_1|$. Now just notice that the inequality $|a_1|\leq |b_1|$ depends on the values of the entries of $Dg^i$. Since the entries of $Dg$ are bounded by above and below by the entries of $Dg^i$, by taking suitable triangular inequalities, it is easy to see that the inequality $|a_1|\leq|b_1|$ also holds for $Dg$, proving that the cones $\mathcal{C}^u$ are also $Dg$-invariant.
		
		Now if we take the norm $\norm{(a,b)}_1:=\max \{|a|,|b|\}$ in $\R^2$, we have that vectors in $\cC^u$ are expanded for the future: if $v=(a,b)\in \cC^u$ then $|a|\leq|b|$ and thus  $\norm{v}_1=|b|$. Since $Dg(v)=(a_1,b_1) \in \cC^u$, this implies that $\norm{Dg(v)}_1=|b_1|$ and we have just proved that $|b|\leq |b_1|$. This implies that $\norm{Dg(v)}_1 \geq \norm{v}_1$ and moreover  $\norm{Dg_{(x,y)}(v)}_1 > \norm{v}_1$ if $(x,y)\neq (0,0)$. 
		
		Next we define the Lyapunov function $V:\R^2\times \R^2\to \R$ given by
		$$ V((x_1,y_1),(x_2,y_2))=(y_1-y_2)^2-(x_1-x_2)^2.
		$$ 
		
		Let $P=(x_1,y_1)$ and $Q=(x_2,y_2)$ be two different points in $\R^2$. Then $V(P,Q)=(y_1-y_2)^2-(x_1-x_2)^2$. Let us denote by $g(P)=(x_1',y_1')$ and $g(Q)=(x_2',y_2')$, then $V(g(P),g(Q))=(y_1'-y_2')^2-(x_1'-x_2')^2$.
		
		Now take the function $\varphi:[0,1]\to \R$ given by $\varphi(t)=\Pi_2(g(tP+(1-t)Q))$, where $\Pi_2:\R^2\to \R$ is the projection in the second coordinate. Notice that $\varphi(0)=y_2'$, $\varphi(1)=y_1'$ and $\varphi$ is differentiable. Then by the Mean value Theorem we have:
		$$|y_2'-y_1'|=|\varphi(0)-\varphi(1)|=|\varphi'(t_0)|=|\Pi_2 \circ Dg_{z_0}(P-Q)|=|Dg_{z_0}(0,y_1-y_2)|>|y_1-y_2|
		$$ for some $z_0$. In the same way we get that $|x_1'-x_2'|<|x_1-x_2|$, and this implies that $\Delta V(P,Q)>0$. This ends the proof of the lemma.  \qed
	\end{proof}
	
	\begin{cor} \label{corejemplo}
		Let $A\in \textnormal{SL}(2,\Z)$ be a hyperbolic matrix and take $\e>0$ sufficiently small. Then there is a diffeomorphism $g:\T^2 \to \T^2$ such that:
		\begin{itemize}
			\item $g(x,y)=A(x,y)$ for every $(x,y) \in B(0,\e)^c$.
			\item $Dg_{(0,0)}=\text{Id}$. 
			\item $g$ belongs to the $C^1$ boundary of Anosov diffeomorphisms.
			\item $g$ is expansive and conjugated to $A$.
		\end{itemize}
	\end{cor}		
	\begin{proof}
		Let $A\in \textnormal{SL}(2,\Z)$ be as in the hypothesis and suppose the eigenvalues of $A$ are $\lambda$ and $\mu$ with: $0<\lambda<1<\mu$. Let $E^s$ be the eigenspace associated to $\lambda$, and let $E^u$ be the eigenspace associated to $\mu$. Then we have that $\R^2=E^s\oplus E^u$. Now for $\e>0$ and indentifying $\R^2$ with $E^s\oplus E^u$ we take the map $g:\R^2\to \R^2$ given by Lemma \ref{lemmixsub}. We can take $\epsilon$ small enough so we can send the map $g$ to the quotient $\T^2=\R^2/\Z^2$ (notice that we are making an abuse of notation between $g$ and its lift to $\R^2$).  
		
		We know that the map $g$ has a Lyapunov function $V:\R^2\times \R^2\to \R$ given by
		$$ V((x_1,y_1),(x_2,y_2))=(y_1-y_2)^2-(x_1-x_2)^2
		$$ 
		which verifies that $\Delta V((x_1,y_1),(x_2,y_2))>0$ for every pair of different (and sufficiently close) points $(x_1,y_1),(x_2,y_2) \in \R^2$. Then by Proposition \ref{LemaLew} the map $g$ is expansive. Notice that we are only interested in what happens in a fundamental domain, so the non-compactness of $\R^2$ is not a problem.  
		
		Finally since $g$ fails to be hyperbolic only at $(0,0)$ we know that $g$ is in the $C^1$ boundary of Anosov diffeomorphisms. Then by Corollary 6.2 in \cite{L} $g$ is conjugated to $A$. Moreover this implies $g$ is infinite expansive in $\R^2$ and in particular, the semiconjugacy $H_g$ is in fact injective and $\Lambda(g)=0$. 
	\end{proof}

	Now that we have the map $g:\R^2\to\R^2$ constructed in Lemma \ref{lemmixsub}, we can build a partially hyperbolic example in $\T^4$ with a center leaf which behaves exactly like $g$.   
	
	\begin{lem} \label{ejendim4}
		Let $A\in \textnormal{SL}(4,\Z)$ be a matrix with four eigenvalues $\lambda^{ss},\lambda,\mu,\mu^{uu}$ such that $ 0<\lambda^{ss}<\lambda<1<\mu<\mu^{uu}$ and take $\e>0$. Then there is a partially hyperbolic diffeomorphism $f:\T^4\to \T^4$ with a splitting of the form $\R^4=E^{ss}_f \oplus E^c_f \oplus E^{uu}_f$, such that $\textnormal{dim}E^c_f=2$ and verifies the following:
		\begin{itemize}
			\item $f(x)=Ax$ for every $B(0,\e)^c$.
			\item $f$ is expansive and conjugated to $A$.
			\item $Df_0|_{E^c_f}=\text{Id}$.
		\end{itemize} 
	\end{lem}
	\begin{proof}
		Let $A\in \textnormal{SL}(4,\Z)$ be a matrix with four eigenvalues $\lambda^{ss},\lambda,\mu,\mu^{uu}$ such that:
		$$ 0<\lambda^{ss}<\lambda<1<\mu<\mu^{uu}.$$ 
		We can assume that in the basis given by the eigenspaces associated to the eigenvalues we have that:  
		$A(x,y,z,t)=(\lambda x, \mu y, \lambda^{ss}z,\mu^{uu}t)$. Take the same functions $\beta_{\lambda}$ and $\beta_{\mu}$ as in Lemma \ref{lemmixsub} and define the map  $f:\R^4\to\R^4$ by:
		$$f(x,y,z,t)=(\lambda x,\mu y, \lambda^{ss} z, \mu^{uu} t) + \rho(w)(\beta_{\lambda}(r)x,-\beta_{\mu}(r)y,0,0)
		$$ where $\rho$ is a bump function supported in $[0,\e]$ and $w=z^2+t^2$. We can take $\epsilon$ small enough so we can send the map $f$ to the quotient $\T^4=\R^4/\Z^4$. 
		
		If $\norm{(x,y,z,t)}\geq \e$ we have that $f=A$. For points with $\norm{(x,y,z,t)} < \e$ the differential of $f$ at a point $(x,y,z,t)$ is:
		
		\begin{equation*} \begin{bmatrix}
				\lambda +\rho(w)\left( \beta_{\lambda}(r) +2x^2 \beta_{\lambda}'(r) \right) & \rho(w)\left(2xy\beta_{\lambda}'(r)\right) & 2xz\rho'(w)\beta_{\lambda}(r) & 2tz\rho'(w)\beta_{\lambda}(r) \\
				-\rho(w)\left(2xy\beta_{\mu}'(r)\right) & \mu -\rho(w)\left( \beta_{\mu}(r)-2y^2 \beta_{\mu}'(r) \right) & -2yz\rho'(w)\beta_{\mu}(r) & 2yt\rho'(w)\beta_{\mu}(r) \\
				0 & 0 & \lambda^{ss} & 0 \\
				0 & 0 & 0 & \mu^{uu}
			\end{bmatrix} 
		\end{equation*}
		In this case the subspace $E^c_f=\{(x,y,0,0)\}$ is $Df$ invariant, and it is quite direct to see that $Df|_{E^c_f}$ is basically $Dg$ like above (we have to deal with the function $\rho$ but is not a problem). In particular we have that: 
		\begin{equation*} \label{difenfix} Df_0=
			\begin{bmatrix}
				1 & 0 & 0 & 0 \\
				0 & 1 & 0 & 0 \\
				0 & 0 & \lambda^{ss} & 0 \\
				0 & 0 & 0 & \mu^{uu}
			\end{bmatrix}.
		\end{equation*}	 
		The strong bundles are not going to be the canonical ones, but if we ask to the strong eigenvalues $\lambda^{ss}$ and $\mu^{uu}$ to be sufficiently far away from 1 (and we can do this by iterating the matrix), the same strong cones for the matrix $A$ are going to be $Df$-invariant.

		Finally in the same way as in Lemma \ref{lemmixsub} and Corollary \ref{corejemplo}, by taking the Lyapunov function $V:\R^4 \times \R^4 \to \R$ given by
		$$ V((x_1,y_1,z_1,t_1),(x_2,y_2,z_2,t_2))=(y_1-y_2)^2 -(x_1-x_2)^2 +(t_1-t_2)^2-(z_1-z_2)^2 
		$$
		it is easy to see that $\Delta V((x_1,y_1,z_1,t_1),(x_2,y_2,z_2,t_2))>0$ for every pair of different (and sufficiently close) points $(x_1,y_1,z_1,t_1),(x_2,y_2,z_2,t_2) \in \R^4$. Then by Proposition \ref{LemaLew} we conclude that $f:\T^4\to \T^4$ is an expansive diffeomorphism. Moreover $f$ is in the $C^1$ boundary of Anosov diffeomorphisms and by Corollary 6.2 in \cite{L}, $f$ is conjugated to $A$ (which implies $\Lambda(f)=0$).
	\end{proof}
	
	\begin{rem}
		The construction in Lemma \ref{ejendim4} can be made with no restriction on the dimensions of the strong subbundles (which were one dimensional in the example above). Indeed, the construction only uses the local isotopy in dimension 2 we made in Lemma \ref{lemmixsub}, and the domination of the external strong subbundles.
	\end{rem}

	\subsubsection{Higher dimensional center bundle}
	
	We now deal with the case where the center bundle has dimension bigger than 2. We are going to treat the case where the center bundle $E^c_A$ is contractive (every center eigenvalue has modulus smaller than 1) since it is the only case that matters to our purposes. The case where the center bundle $E^c_A$ is expanding (every center eigenvalue has modulus bigger than 1) is completely analogous. We remark that the proof works as well for the case where the center is two dimensional. 
	
	In Lemma \ref{lemahdc} we construct the example in $\R^k$ and then in Lemma \ref{lemgeneral} we insert the example as the center leaf of a partially hyperbolic diffeomorphism like we did before. 
	
	\begin{lem} \label{lemahdc}
		Let $A\in \textnormal{M}_{k\times k}(\Z)$ be a diagonal matrix with $k$ eigenvalues such that $$ 0<\lambda_1 \leq \dots \leq \lambda_k <1.
		$$ Then for every $\e>0$ there is a diffeomorphism $g:\R^k\to \R^k$ such that: 
		\begin{itemize}
			\item $g(x)=Ax$ for every $x\in B(0,\e)^c$.
			\item $\norm{Dg_x}<1$ for every $x\neq 0$.
			\item $Dg_0=\text{Id}$.
		\end{itemize} 
	\end{lem}
	\begin{proof}
		Take a matrix $A\in M_{k\times k}(\Z)$ as above. Then the eigenvalues of $A$ verify that:
		$$ 0<\lambda_1 \leq \dots \leq \lambda_k <1.
		$$
		Fix $\e>0$ small, and take $\lambda$ such that $\lambda_k < \lambda < 1$. In particular $\lambda > \lambda_j$ for every $j=1,\dots,k$. Now take a number $c\in (0,\e)$ such that: $c<\frac{1-\lambda}{k(\lambda-\lambda_1)}$.
		Now, for this $c>0$ take the function $\beta$ given by Lemma \ref{lemauxfunction} for $b=1$. In particular, the function $\beta$ verifies:
		\begin{itemize}
			\item $\beta$ is $C^{\infty}$ and decreasing. 
			\item $\beta$ is supported in $[0,\e]$. 
			\item $\beta(0)=1$.
		\end{itemize}
		Moreover, we can ask for $\beta$ to be equal to 1 in a small interval $[0,\delta]$. Now we can define the map $g_1:\R^k \to \R^k$ by
		$$g_1(x_1,\dots,x_k)=A(x_1,\dots,x_k)+\beta(r)((\lambda-\lambda_1)x_1,\dots, (\lambda-\lambda_k)x_k)
		$$ where $r=x_1^2+\dots +x_k^2$. Since $supp(\beta)\subseteq [0,\e]$ we have that if $\norm{x}>\e$ then $g_1=A$. The differential of $g_1$ in a point $x$ is:
		\begin{equation*} D(g_1)_x=
			\begin{bmatrix}
				\lambda_1 & \ & \ & \ \\
				\ & \ddots & \ & \ \\
				\ & \ & \ddots & \ \\
				\ & \ & \ & \lambda_k
			\end{bmatrix}+ \beta(r)
			\begin{bmatrix}
				\lambda-\lambda_1 & \ & \ & \ \\
				\ & \ddots & \ & \ \\
				\ & \ & \ddots & \ \\
				\ & \ & \ & \lambda-\lambda_k
			\end{bmatrix}+M(x)
		\end{equation*} where $M(x)$ is the matrix given by
		\begin{equation*}M(x)=2\beta'(r)
			\begin{bmatrix}
				(\lambda-\lambda_1)x_1^2 & (\lambda-\lambda_1)x_1x_2 & \dots & (\lambda-\lambda_1)x_1x_k \\
				(\lambda-\lambda_2)x_1x_2 & (\lambda-\lambda_2)x_2^2 & \dots & (\lambda-\lambda_2)x_2x_k \\
				\vdots & \vdots & \ddots & \vdots \\
				(\lambda-\lambda_k)x_1x_k & (\lambda-\lambda_k)x_2x_k & \dots & (\lambda-\lambda_k)x_k^2
			\end{bmatrix}.	
		\end{equation*}
		In particular, since $\beta(0)=1$ we have that $D(g_1)_0=A+\beta(0)(\lambda \textit{Id} -A)+M(0)=\lambda \textit{Id}$. Now take a point $x\in \R^k$ and a vector $v\in \R^k$, then we have that:
		$$D(g_1)_x(v)=Av+\beta(r)(\lambda \textit{Id}-A)v+M(x)v. 
		$$
		Assume the vector $v$ is equal to $v=(a,\dots,a)\in \R^k$ for a given $a\in \R$, and denote by $D(g_1)_x(v)=(a_1,\dots,a_k)$. If we prove that $|a_j|<|a|$ for every $j=1,\dots,k$, we obtain that $D(g_1)_x$ is a contraction (by taking the norm of the maximum). Let's take a look at the first coordinate $a_1$: 
		\begin{equation*}
			a_1=a\left(\lambda_1+\beta(r)(\lambda-\lambda_1)+2\beta'(r)(\lambda-\lambda_1)\sum_{j=1}^{k}x_1x_j\right).
		\end{equation*}
		By taking absolute value, and applying the triangular inequality we obtain:
		\begin{equation*}
			|a_1|\leq|a|\left(|\lambda_1+\beta(r)(\lambda-\lambda_1)|+2|\beta'(r)(\lambda-\lambda_1)|\sum_{j=1}^{k}|x_1x_j|\right).
		\end{equation*}
		Notice that $0\leq (x_i+x_j)^2=x_i^2+x_j^2+2x_ix_j$ and this implies $2|x_ix_j|\leq x_i^2+x_j^2\leq r$. Now recall that: $|\beta'(r)|\leq \frac{c}{r} < \frac{1-\lambda}{k(\lambda-\lambda_1)r}
		$ and in particular we have that 
		$$ 2|\beta'(r)(\lambda-\lambda_1)|\sum_{j=1}^{k}|x_1x_j| < 1-\lambda.
		$$
		This implies that:
		\begin{equation*}
			|a_1|<|a|\left(|\lambda_1+\beta(r)(\lambda-\lambda_1)|+ 1-\lambda \right).
		\end{equation*}
		Since $\beta$ is a decreasing function, we have that $1=\beta(0)\geq \beta (r)$ and then: 
		\begin{equation*}
			|a_1|<|a|\left(|\lambda_1+\beta(r)(\lambda-\lambda_1)|+ 1-\lambda \right) \leq |a|\left(|\lambda_1+(\lambda-\lambda_1)|+ 1-\lambda \right)=|a|.
		\end{equation*}
		The exact same calculation shows that $|a_j|<|a|$ for every $j=1,\dots,k$. This shows that $D(g_1)_x$ is a contraction (for the norm of the maximum) for every $x\in \R^k$ and in particular, $g_1$ is expansive. Notice that since $\beta(r)=1$ for every $r\in [0,\delta]$, we have that $g_1(x)=\lambda x$ for every $x\in B(0,\delta)$. Now take the function $h:\R^k\to \R^k$ given by $h(x)=(1-r)x$ where $r=\norm{x}^2$ and consider a bump function $\rho:[0,+\infty) \to \R$ such that:
		\begin{itemize}
			\item $\rho(t)=1$ for every $t\in [0,\delta/2]$. 
			\item $\rho(t)=0$ for every $t\geq \delta$.
		\end{itemize} 
		Now we define $g:\R^k\to \R^k$ given by the equation:
		\begin{equation*}
			g(x)=\rho(r)h(x)+(1-\rho(r))g_1(x)
		\end{equation*} where  $r=\norm{x}^2$. 
		The first observation is that if $\norm{x}\geq \delta$ then $g(x)=g_1(x)=Ax$. On the other hand, if $r=\norm{x}^2\leq \delta$ then we have that $g_1(x)=\lambda x$ and therefore $$g(x)=\rho(r)h(x)+(1-\rho(r))\lambda x=[\rho(r)(1-r)+(1-\rho(r))\lambda]x$$ and the function $g$ is radial. Denote by $\alpha(r):=\rho(r)(1-r)+(1-\rho(r))\lambda$, then it is direct to see that $$ \alpha(r)=\rho(r)(1-r)+(1-\rho(r))\lambda=\rho(r)(1-r-\lambda)+\lambda \leq 1-r.$$
		As a result $g$ sends every sphere of radius $R$ to a sphere of radius $\alpha(R)R$ which is strictly smaller than $R$. This implies that $Dg_x$ is a contraction for every $x\in B(0,\delta)$. To see this, just notice that given $x\in \R^k$ we have that $T_x\R^k=T_x S_{\norm{x}}+ \langle x \rangle$ where $S_{\norm{x}}$ is the sphere centered at 0 of radius $\norm{x}$. The same happens with $g(x)$, i.e. $T_{g(x)}\R^k=T_{g(x)} S_{\alpha(\norm{x})\norm{x}}+ \langle x \rangle$ and the differential of $g$ at $x$ restricted to this subspace is exactly $$Dg_x|_{T_xS_{\norm{x}}}=\alpha(\norm{x})\textit{Id}$$ which is a contraction. The other direction $\langle x \rangle$ is exactly the same, and there therefore $Dg_x$ is a contraction. To finish the proof just observe that if $r=\norm{x}^2<\delta/2$ we have that $g(x)=(1-r)x$ and in particular $Dg_0=\textit{Id}$. 
	\end{proof}
	
	Now by applying the same trick as in Lemma \ref{ejendim4} (with a suitable bump function) we can embed the example above as the center leaf of a higher dimensional manifold. We thus obtain the following result whose proof we omit since it is exactly the same as the one of the lemma we just mentioned.
	
	\begin{lem} \label{lemgeneral}
		Let $A\in \textnormal{SL}(n,\Z)$ be a symmetric matrix with a splitting of the form $\R^n=E^{ss}_A\oplus E^c_A \oplus E^{uu}_A$ s.t. $\textnormal{dim}E^c_A=k$ and $E^c_A$ is the eigenspace associated to the eigenvalues $0<\lambda_1\leq \dots \leq \lambda_k<1$. 
		Then for every $\e>0$ small, there is $f\in \cPH_A(\T^n)$ such that: \begin{itemize}
			\item $f(x)=Ax$ for every $x\in B(0,\e)^c$.
			\item $f$ is expansive and conjugated to $A$.
			\item $Df_0|_{E^c_f}=\text{Id}.$
		\end{itemize}	
	\end{lem}
	
	We now proceed to finish the proof of Theorem \ref{teoejemixed}. We begin with the case where $n=4$ since it is quite direct for our previous results and illustrates the general ideas. We then prove the general case.
	
	\subsection{Proof for case $n=4$}
	
	Take a matrix $A\in \textnormal{SL}(4,\Z)$ with four eigenvalues $\lambda^{ss},\lambda,\mu,\mu^{uu}$ such that:
	$$ 0<\lambda^{ss}<\lambda<1<\mu<\mu^{uu}.$$ 
	This induces a splitting of the form $\R^4=E^{ss}\oplus E^{ws}\oplus E^{wu}\oplus E^{uu}$ and we take the center bundle as $E^c=E^{ws}\oplus E^{wu}$. We can assume that in the basis given by the eigenspaces associated to the eigenvalues we have that:  
	$$A(x,y,z,t)=(\lambda x, \mu y, \lambda^{ss}z,\mu^{uu}t).$$ Moreover, we can assume that the linear Anosov $A$ has four different fixed points: $Fix(A)=\{p_0,p_1,p_2,p_3\}$ (we are making an abuse of notation here, by calling $A$ instead of $f_A$, the induced map in the torus). We just have to iterate the matrix a few times in order to have four different fixed points.
	
	Now notice that the procedure made in Subsection \ref{ssexpda} works as well. First, for every fixed point $p_j$ (with $j=0,1,2$) take a small neighbourhood $U_j$ (notice that we are not going to perturb $p_3$ since it already has index 2). We can take them small enough to be disjoint. Second, just notice that the isotopy procedure we made in Lemma \ref{lemmixsub} is only local, and therefore it can be applied in different disjoint neighbourhoods. Hence the same proof as in Lemma \ref{ejendim4} shows that we can make an isotopy whose support is contained in $U_0 \cup U_1 \cup U_2$ in order to get a partially hyperbolic diffeomorphism $f_1: \T^4\to \T^4$ such that:
	\begin{itemize}
		\item $f_1(x)=Ax$, for every $x\in (U_0 \cup U_1 \cup U_2)^c$.
		\item $f_1$ is hyperbolic outside $Fix(f_1)=\{p_0,p_1,p_2\}$.
		\item $D(f_1)_{p_j}|_{E^c_{f_1}(p_j)}=\textit{Id}$ for $j=0,1,2$.
	\end{itemize}   
	Since $f_1$ is hyperbolic outside $Fix(f_1)=\{p_0,p_1,p_2\}$, once again we can find an appropriate Lyapunov function in order to apply Proposition \ref{LemaLew} and get that $f_1$ is expansive. Since $f_1$ is in the $C^1$ boundary of Anosov diffeomorphisms, by Corollary 6.2 in \cite{L} we have that $f_1$ is conjugated to $A$ which implies $\Lambda(f_1)=0$. 
	
	The first point shows that $f_1$ is SH-saddle of index (1,1). To see this, we just have to observe that the same proof of Lemma \ref{DAesSH} shows that $f_1$ have the SH-Saddle property as well. In that proof, the only property we use is the fact that, for a point $p$ outside $U$ and given a small $\delta>0$, there is always a point $p_1$ such that $$\cW^{uu}_{f_1}(p_1,\delta)\subset f_1(\cW^{uu}_{f_1}(p,\delta))\cap (U^c)$$ and by an induction argument we find a point whose forward orbit never meets $U$, and the same happens for the past. Hence $f_1$ is SH-Saddle of index (1,1). 
	By the same arguments, by taking the strong bundles $E^{ss}$ and $E^{uu}$ sufficiently contractive and expanding, and taking the neighbourhoods $U_j$ sufficiently small, we also have this property, i.e. for every point $p$ outside $U_0\cup U_1\cup U_2$, and given a small $\delta>0$, there is always a point $p_1$ such that $$\cW^{uu}_{f_1}(p_1,\delta)\subset f_1(\cW^{uu}_{f_1}(p,\delta))\cap (U_0\cup U_1\cup U_2)^c.$$
	Then, we can find a point that never meets $U_0\cup U_1\cup U_2$ for the future, and the same for the past. In short, $f_1$ has the (1,1) SH-Saddle property.  
	
	Now since $f_1$ is SH-Saddle we have that $\rho(f_1)>0$, and by expansiveness we also have $\Lambda(f_1)=0$. Then a direct application of Theorem \ref{teorobtran} shows that $f_1$ is $C^1$ robustly transitive. Let's call $\cU_1$ to the $C^1$ neighbourhood of $f_1$ such that every $h\in \cU_1$ is transitive. 
	
	To end the proof of the theorem, we are going to change the indexes of the fixed points $p_0$ and $p_1$, and to put a complex eigenvalue in $p_2$. First take the two matrixes $B_0$ and $B_1$ given by: 
	
	\begin{equation*} B_0=
		\begin{bmatrix}
			1-\eta & 0 & 0 & 0 \\
			0 & 1-\eta & 0 & 0 \\
			0 & 0 & \lambda^{ss} & 0 \\
			0 & 0 & 0 & \mu^{uu}
		\end{bmatrix}
		\ \ B_1=
		\begin{bmatrix}
			1+\eta & 0 & 0 & 0 \\
			0 & 1+\eta & 0 & 0 \\
			0 & 0 & \lambda^{ss} & 0 \\
			0 & 0 & 0 & \mu^{uu}
		\end{bmatrix}.
	\end{equation*} 
	Then for $\eta$ sufficiently small we have that the matrices $B_0$ and $B_1$ are $\e$ close to $D(f_1)_{p_0}$ and $D(f_1)_{p_1}$ respectively.
	Now in order to mix the two center subbundles, take the matrix $B_2$ with the form:
	\begin{equation*} B_2=
		\begin{bmatrix}
			a & b & 0 & 0 \\
			-b & a & 0 & 0 \\
			0 & 0 & \lambda^{ss} & 0 \\
			0 & 0 & 0 & \mu^{uu}
		\end{bmatrix}
	\end{equation*}  
	where $a\pm ib$ are the complex eigenvalues of $B_2$. It is possible to take $a$ and $b$ such that $a$ is close to 1, $b$ is close to 0 (the modulus of $a\pm ib$ can be smaller, bigger or equal to 1 for our purposes). For suitable values of $a$ and $b$ we can assure that $B_2$ is $\e$ close to $D(f_1)_{p_2}$.  
	Then by Franks Lemma \cite{Fr1}, there is a diffeomorphism $f\in \cU_1$ such that:
	\begin{itemize}
		\item $f(x)=f_1(x)$ for every $x\in (U_0 \cup U_1 \cup U_2)^c$.
		\item $f(p_j)=f_1(p_j)=p_j$ for $j=0,1,2$.
		\item $Df_{p_j}=B_j$ for $j=0,1,2$.
	\end{itemize}
	In particular $index(p_0)=3$ and $index(p_1)=1$ (recall that $index(p_3)=2$). Since $Df_{p_2}$ has a center complex eigenvalue, the center bundle of $f$ can not be decomposed into two 1-dimensional subbundles. To sum up, the map $f:\T^4 \to \T^4$ is a $C^1$ robustly transitive derived from Anosov diffeomorphism, and verifies all the properties of Theorem \ref{teoejemixed}.

	\begin{rem}
		All the examples constructed in Lemmas \ref{lemmixsub}, \ref{ejendim4}, \ref{lemahdc} and \ref{lemgeneral} are not generic, since they are not hyperbolic but conjugated to its linear part which is a hyperbolic matrix. In fact, they are in the $C^1$ boundary of Anosov diffeomorphisms.
	\end{rem}
	
	\subsection{Proof of the general case}
	For the proof of the general case we proceed like we did above. Let $A\in \textnormal{SL}(n,\Z)$ be a hyperbolic symmetric matrix with a splitting of the form:
	\begin{equation*}
		\R^n=E^{ss}_A\oplus E^{ws}_A \oplus E^{wu}_A \oplus E^{uu}_A
	\end{equation*} 
	where we take $E^c_A:=E^{ws}_A \oplus E^{wu}_A$ as the center bundle. Since the matrix $A$ is symmetric we know the subbundles $E^{ws}_A$ and $E^{wu}_A$ can be decomposed into 1-dimensional subbundles, i.e.:
	\begin{equation*}
		\R^n=E^{ss}_A\oplus E^{ws}_1\oplus \dots \oplus E^{ws}_m \oplus E^{wu}_1 \oplus \dots \oplus E^{wu}_l \oplus E^{uu}_A
	\end{equation*} where $E^{w*}_j$ is the eigenspace associated to the eigenvalue $\lambda^{*}_j$ for $*=s,u$. In particular the eigenvalues verify:
	$$\lambda^s_1\leq \dots \leq \lambda^s_m<1<\lambda^u_1 \leq \dots \leq \lambda^u_l.$$
	In short $m=\text{dim}E^{ws}_A$, $l=\text{dim}E^{wu}_A$ and $k=\text{dim}E^c_A=m+l$.

	Notice that 0 is a fixed point of $A$ and $index(0)=\text{dim}E^{ss}_A+m$. Now by iterating the matrix if necessary we can take $k=m+l$ different fixed points of $A$ (here we are making an abuse of notation once again), $Fix(A)=\{p_1,\dots, p_m,q_1,\dots,q_l\}$. For every $j=1,\dots,m$ take a neighbourhood $U_j$ of $p_j$, and for every $j=1,\dots,l$ take a neighbourhood $V_j$ of $q_j$. We can assume that they are small enough to be disjoint. 
	
	Like before, we proceed like in Subsection \ref{ssexpda}. Notice that the isotopies we made in that subsection were only local. Therefore a direct application of Lemma \ref{lemgeneral} implies that there is a partially hyperbolic diffeomorphism $g:\T^n\to \T^n$ with a splitting of the form 
	$$T\T^n=E^{ss}_g\oplus E^{ws}_g\oplus E^{wu}_g \oplus E^{uu}_g$$
	where $\text{dim}E^{*}_A=\text{dim}E^{*}_g$ for $*=ss,ws,wu,uu$, and moreover:
	\begin{itemize}
		\item $g(x)=Ax$ for every $x\in \left( U_1 \cup \dots \cup U_m \cup V_1 \cup \dots \cup V_l\right)^c$.
		\item $g$ is hyperbolic outside $Fix\{g\}$. 
		\item $Dg_{p_j}|_{E^{ws}_g}=\textit{Id}$ for every $j=1,\dots,m$.
		\item $Dg_{q_j}|_{E^{wu}_g}=\textit{Id}$ for every $j=1,\dots,l$.
	\end{itemize}
	Once again, by taking the neighborhoods $U_i$ and $V_j$ sufficiently small, the same argument in Lemma \ref{DAesSH} implies $g$ has the SH-Saddle property of index $(m,l)$. The second point above says that $g$ is expansive and in the $C^1$ boundary of Anosov diffeomorphisms. Then by Corollary 6.2 in \cite{L}, $g$ is conjugated to $A$ which implies $\Lambda(g)=0$. Then by Theorem \ref{teorobtran} (or Corollary \ref{coroSHtransitivo}) we have that $g$ is $C^1$ robustly transitive. Let $\cU$ be the $C^1$ neighbourhood of $g$ such that every $h\in \cU$ is transitive, and let $\e>0$ be such that $B_{C^1}(g,\e)\subset \cU$. 
	
	Now for this $\e$, take $m$ hyperbolic matrices $B_1,\dots,B_m$ which are $\e$ close to $Dg_{p_1},\dots, Dg_{p_m}$, and such that $index(B_j)=\text{dim}E^{ss}_A+j$. Notice that we can always have these matrices since $Dg_{p_j}|_{E^{ws}_g}=\textit{Id}$ for every $j=1,\dots,m$. In the same way we can take $l$ hyperbolic matrices $C_1,\dots,C_l$ which are $\e$ close to $Dg_{q_1},\dots, Dg_{q_l}$, and such that $index(C_j)=\text{dim}E^{ss}_A+m+j$. 
	
	By applying Franks Lemma \cite{Fr1} once again, we know there is a partially hyperbolic diffeomorphism $f \in \cU \cap \cPH_A(\T^n)$ such that:
	\begin{itemize}
		\item $f(x)=g(x)=Ax$ for every $x\in \left( U_1 \cup \dots \cup U_m \cup V_1 \cup \dots \cup V_l\right)^c$.
		\item $f(p_j)=p_j$ for every $j=1,\dots,m$.
		\item $f(q_j)=q_j$ for every $j=1,\dots,l$.
		\item $Df_{p_j}=B_j$ for every $j=1,\dots,m$.
		\item $Df_{q_j}=C_j$ for every $j=1,\dots,l$.
	\end{itemize}
	In particular, we have $k+1$ fixed points (we are including 0 here) with indexes going from $\text{dim}E^{ss}_A$ to $\text{dim}E^{ss}_A+k$. 
	
	To end the proof we have to mix the center subbundles $E^{ws}_m$ and $E^{wu}_1$. To do this, we just have to take another different fixed point $p$ and apply the same isotopy as in Lemma \ref{lemmixsub}. This way the splitting of $f$ is not coherent with the hyperbolic splitting of $A$. If we want to make the entire center undecomposable, we can take extra fixed points and make suitable local perturbations by adding complex eigenvalues. We thus obtain our example and we finish the proof of Theorem \ref{teoejemixed}.

	\subsection{Another kind of examples}
	In this subsection we are going to present two additional examples of $C^1$ robustly transitive partially hyperbolic diffeomorphisms. These examples are in a sense similar to the ones we saw in Theorem \ref{teoejemixed} but with a different flavor.
	
	We begin by introducing an example which originally appears in \cite{L} (see also \cite{CeLe}), that will be used in the construction of both $C^1$ robustly transitive partially hyperbolic diffeomorphisms.
	
	\subsubsection{Example 0} \label{Ex0}
	Let $A$ be the hyperbolic matrix
	$A=\begin{bmatrix} 
		2 & 1 \\ 1 & 1
	\end{bmatrix}.$
	Take $c\in [0,1]$ and consider the family of diffeomorphisms $f_c:\T^2\to \T^2$ given by
	$$f_c(x,y)=\left(2x+y-\frac{c}{2\pi}\sin(2\pi x)\,,\,x+y-\frac{c}{2\pi}\sin(2\pi x)\right).$$
	The differential of $f_c$ at a point $(x,y)$ is equal to
	$$Df_c=\begin{bmatrix} 
		2-c\cos(2\pi x) & 1 \\ 1-c\cos(2\pi x) & 1
	\end{bmatrix}.
	$$
	Notice that $det(Df_c)=1$ for every $c\in [0,1]$ and every $(x,y)\in \T^2$, therefore $f_c$ is a conservative diffeomorphism for every $c\in [0,1]$. We also observe that the trace $tr(Df_c)=3-c\cos(2\pi x)>2$ if $c<1$, which implies $f_c$ is Anosov (notice in particular that $f_0=A$) and by the structural stability of Anosov diffeomorphisms, $f_c$ is conjugated to $A$ . 
	
	When $c=1$ the map $f_1$ is not uniformly hyperbolic since for $x=0$ we have
	\begin{equation}\label{matrixJ}
		Df_1=\begin{bmatrix} 
			1 & 1 \\ 0 & 1
		\end{bmatrix}
	\end{equation}
	and therefore every point of the form $(0,y)$ has a non-hyperbolic differential, and moreover does not admit any invariant subbundle. For points $(x,y)$ with $x\neq 0$, $Df_1$ is a hyperbolic matrix since the trace is bigger than 2. 
	
	To sum up, $f_1$ is a conservative non-Anosov diffeomorphism on $\T^2$. Moreover by taking the Lyapunov function $V:\R^2\times\R^2 \to \R$ given by
	$$ V((x_1,y_1),(x_2,y_2))=-(y_2-y_1)((y_2-x_2)-(y_1-x_1))
	$$ it is easy to see that 
	$$\Delta V((x_1,y_1),(x_2,y_2))=(x_2-x_1)^2+(y_2-y_1)^2-(x_2-x_1) \left( \frac{\sin (2\pi x_2)-\sin (2\pi x_1)}{2\pi} \right)
	$$ which is positive by the mean value theorem, except perhaps at finite points. By taking $\gamma>0$ sufficiently small (to remove these finite points), we have that $\Delta V((x_1,y_1),(x_2,y_2))>0$ if $0<d((x_1,y_1),(x_2,y_2))\leq \gamma$ and by Proposition \ref{LemaLew} we know $f_1$ is $\gamma$-expansive. 
	
	Finally notice that if $c\to 1$ then $f_c \to f_1$ in the $C^1$ topology, and then by Corollary 6.2 in \cite{L} we have that $f_1$ is conjugated to $f_c$ for any $c$ sufficiently close to 1. Since $f_c$ is conjugated to $A$ for any $c<1$, we conclude $f_1$ is conjugated to $A$. 
	
	\subsubsection{Example 1}
	Take an Anosov diffeomorphism with two fixed points, for example $A^2:\T^2\to \T^2$ where $A$ is the matrix in \textit{Example 0}. Denote by $p_1$ and $q_1$ the two fixed points of $A^2$. Now for $i=1,2$ take continuous functions $c_i:\T^2 \to \R$ such that:
	\begin{itemize}
		\item $0\leq c_i(z) \leq 1$ for every $z\in \T^2$.
		\item $c_1(z)=1$ if and only if $z=p_1$.
		\item $c_2(z)=1$ if and only if $z=q_1$.
	\end{itemize} 	
	Then we define the map $F:\T^2 \times \T^2 \times \T^2 \to \T^2 \times \T^2 \times \T^2$ by
	$$F(z_1,z_2,z_3)=(A^2z_1,f_{c_1(z_1)}(z_2),f_{c_2(z_1)}(z_3))
	$$
	where the maps $f_c$ come from \textit{Example 0} in \ref{Ex0}. Then the differential of $F$ at a point $(z_1,z_2,z_3)$ is equal to
	\begin{equation*}
		DF=\begin{bmatrix}
			A^2 & 0 & 0 \\ 
			\frac{\partial}{\partial z_1} f_{c_1^{}(z_1)}  & Df_{c_1^{}(z_1)} & 0 \\ 
			\frac{\partial}{\partial z_1} f_{c_2^{}(z_1)} & 0 & Df_{c_2^{}(z_1)}\\ 
		\end{bmatrix}.
	\end{equation*}
	If we take $\frac{\partial}{\partial z_1} f_{c_1^{}(z_1)}$ and $\frac{\partial}{\partial z_1} f_{c_2^{}(z_1)}$ sufficiently small (by taking $c_1$ and $c_2$ sufficiently small), then $F$ is a conservative partially hyperbolic diffeomorphism with center leaves $\cW^c_F(z_1,z_2,z_3)=\{z_1\}\times \T^2\times \T^2$. 
	The points $P=(p_1,0,0)$ and $Q=(q_1,0,0)$ are fixed by $F$ and they are not hyperbolic since
	$$DF_P^{}=\begin{bmatrix}
		A^2 & 0 & 0 \\ 
		\frac{\partial}{\partial z_1} f_{c_1^{}(p_1)}  & Df_1^{}   & 0 \\ 
		\frac{\partial}{\partial z_1} f_{c_2^{}(p_1)} & 0 &  Df_{c_2^{}(p_1)}\\ 
	\end{bmatrix}\ \ \ \ \text{and} \ \ \ \  
	DF_Q^{}=\begin{bmatrix}
		A^2 & 0 & 0 \\ 
		\frac{\partial}{\partial z_1} f_{c_1^{}(q_1)}  & Df_{c_1^{}(q_1)} & 0 \\ 
		\frac{\partial}{\partial z_1} f_{c_2^{}(q_1)} & 0 & Df_1^{} \\ 
	\end{bmatrix}
	$$ 
	where $Df_1^{}$ is like in \eqref{matrixJ}.	Moreover, the region of the manifold where hyperbolicity fails is contained in $\{p_1,q_1 \} \times \T^2\times \T^2$. Now notice that the strong stable/unstable bundles are almost horizontal. Then every strong stable/unstable leaf of sufficiently large (and uniform) length, will be transversal to the set $\{p_1,q_1 \} \times \T^2\times \T^2$, and then as in the proof of  Lemma \ref{DAesSH}, for small neighborhoods $U_P^{}$ and $U_Q^{}$ of $P$ and $Q$ respectively, we get points in unstable leaves that never enters $U_{P}^{}\cup U_Q^{}$ for the future, and points in stable leaves that never enters $U_{P}^{}\cup U_Q^{}$ to the past. This implies that $F$ has the SH-Saddle property of index $(2,2)$.
	
	Then as we did in the previous examples, we can find a suitable Lyapunov function to see that $F$ is expansive and conjugated to $(A^2,A,A)$ which implies $\Lambda(F)=0$ (see \cite{L} where a similar example is constructed). Then by Theorem \ref{teorobtran} (or Corollary \ref{coroSHtransitivo}) we know $F$ is $C^1$ robustly transitive. If we want to change the indexes of $P$ and $Q$, we can just apply Franks lemma \cite{Fr1} as we already showed. 
	
	
	\subsubsection{Example 2}
	For $v\in \R^2$ let $T_v:\R^2\to \R^2$ be the translation $T_v(w)=w+v$. Then for every $v\in \R^2$ we define $G_v:\R^2\to \R^2$ by $G_v=T_v \circ f_1 \circ T_v^{-1}$ where $f_1$ is like the \textit{Example 0} we built in Subsection \ref{Ex0}. It is easy to see that $G_v$ is $\Z^2$-invariant, and then induces a diffeomorphism $g_v:\T^2\to\T^2$. By the chain rule we have $D(g_v)_w=D(f_1)_{w-v}$ and therefore $g_v$ is conjugated to $A$. 
	
	Now we define $f:\T^2\times \T^2 \to \T^2\times \T^2$ by $f(v,w)=(A^2v,g_v(w))$.
	Then it is easy to see that $f$ is a partially hyperbolic diffeomorphism where the center leaves are the fibers $\cWc_f(v,w)=\{v\}\times \T^2$. Observe that we define $f$ by $A^2$ in the base, in order to get domination. It is not difficult to see that $f$ is conjugated to $A^2\times A$. Now we want to see that $f$ has the SH-saddle property of index $(1,1)$. To see this, observe that since $D(g_v)_w=D(f_1)_{w-v}$, the points with non-hyperbolic behavior lay on the set $$\Delta=\{(x,y,x,t)\in \T^2\times \T^2\}.$$ Then as in \textit{Example 1}, we observe that strong stable/unstable bundles are almost horizontal, and therefore, every strong stable/unstable leaf of sufficiently large (and uniform) length, will be transversal to the set $\Delta$. Then by applying the same techniques of Lemma \ref{DAesSH} we get points that never enters a small neighborhood of $\Delta$ for the past and the future respectively, and this implies the SH-Saddle property of index $(1,1)$.
	
	To sum up, $f:\T^2\times \T^2 \to \T^2\times \T^2$ is a partially hyperbolic diffeomorphism, not Anosov, expansive, conjugated to $A^2\times A$ (which implies $\Lambda(f)=0$), and with the SH-saddle property of index $(1,1)$. Then by Theorem \ref{teorobtran} (Corollary \ref{coroSHtransitivo}) $f$ is $C^1$ robustly transitive. From here, we can take any number of fixed points and to perform the same perturbations and apply Franks lemma \cite{Fr1} in order to change the indexes of fixed points. We remark that this example is different to the ones we built in Theorem \ref{teoejemixed} and \textit{Example 1}, since the points of the manifolds where hyperbolicity fails is not localized in small neighborhoods of fixed points.

		

\small
\vspace{0.5cm}
\begin{flushleft}
	\textsc{Luis Pedro Pi\~neyr\'ua}\\
	IMERL, Facultad de Ingenier\'ia\\
	Universidad de la Rep\'ublica, Montevideo, Uruguay\\
	email: \texttt{lpineyrua@fing.edu.uy} 
\end{flushleft}

\end{document}